\documentclass{article}
\usepackage[utf8]{inputenc}

\usepackage[left=1.25in, right=1.25in, top=1in, bottom=1in]{geometry}
\usepackage{amsmath}
\usepackage{amsfonts}
\usepackage{amssymb}
\usepackage{amsthm}
\usepackage{mathtools}
\usepackage[pdfencoding=auto]{hyperref}
\usepackage{xcolor}
\usepackage{graphicx}
\usepackage{caption}
\usepackage{subcaption}
\usepackage{float}
\usepackage{bbm}
\usepackage{bm}
\usepackage{caption}
\usepackage{subcaption}
\usepackage{flafter}
\usepackage{ifthen}
\usepackage{multirow}
\usepackage{tikz-cd}

\usepackage{etoolbox}
\tracingpatches
\makeatletter
\patchcmd{\@setref}{\bfseries ??}{\bfseries\color{red} label}{}{}
\makeatother

\hypersetup{colorlinks=true,
	linkcolor=blue,
	filecolor=magenta,      
	urlcolor=cyan,
	citecolor=teal}

\renewcommand{\qedsymbol}{${}_\blacksquare$}

\newtheorem{defi}{Definition}[section]
\theoremstyle{definition}
\newtheorem{theorem}[defi]{Theorem}
\newtheorem*{theorem*}{Theorem}
\newtheorem{lemma}[defi]{Lemma}
\newtheorem{conj}[defi]{Conjecture}
\newtheorem{coro}[defi]{Corollary}

\newtheorem{examples}[defi]{Examples}

\newcommand{\Ggraph}[2]{\mathcal{G}_{#1}\left(#2\right)}
\renewcommand{\hom}{\text{Hom}}
\renewcommand{\epsilon}{\varepsilon}
\newcommand{\R}{\mathbb{R}}
\newcommand{\Z}{\mathbb{Z}}
\renewcommand{\P}[1]{\mathbb{P}\left[#1\right]}
\newcommand{\diam}[1]{\text{diam}\left(#1\right)}

\newcommand{\pois}[1]{\text{Pois}\left(#1\right)}
\newcommand{\calB}{\mathcal{B}}
\newcommand{\calC}{\mathcal{C}}
\newcommand{\calO}{\mathcal{O}}
\renewcommand{\S}{\mathbb{S}}

\newcommand{\N}{\mathbb{N}}
\newcommand{\Q}{\mathbb{Q}}

\renewcommand{\phi}{\varphi}
\newcommand{\cov}[2]{\text{cov}_{#1}\left(#2\right)}
\newcommand{\ind}[2]{\text{ind}_{#1}\left(#2\right)}
\newcommand{\interior}[1]{\text{int}\left(#1\right)}
\newcommand{\1}{\mathbbm{1}}
\renewcommand{\bar}{\overline}

\newcommand{\dist}[2][]{%
	\ifthenelse{ \equal{#1}{} }
	{\mathbbm{d}\left(#2\right)}
	{\mathbbm{d}_{#1}\left(#2\right)}
}

\newcommand{\vol}[2][]{\text{Vol}_{#1}\!\left(#2\right)}

\title{Generalized Borsuk Graphs}
\author{Francisco Martinez-Figueroa}
\date{\today}

\begin{document}
	
	\maketitle
	
	\begin{abstract}
		Given a finite group $G$ acting freely on a compact metric space $M$, and $\epsilon>0$, we define the $G$-Borsuk graph on $M$ by drawing edges $x\sim y$ whenever there is a non-identity $g\in G$ such that $\dist{x,gy}\leq\epsilon$. We show that when $\epsilon$ is small, its chromatic number is determined by the topology of $M$ via its $G$-covering number, which is the minimum $k$ such that there is a closed cover $M=F_1\cup\dots\cup F_k$ with $F_i\cap g(F_i)=\emptyset$ for all $g\in G\setminus\{\1\}$. We are interested in bounding this number. We give lower bounds using $G$-actions on Hom-complexes, and upper bounds using a recursive formula on the dimension of $M$. We conjecture that the true chromatic number coincides with the lower bound, and give computational evidence. We also study random \mbox{$G$-Borsuk} graphs, which are random induced subgraphs. For these, we compute thresholds for $\epsilon$ that guarantee that the chromatic number is still that of the whole $G$-Borsuk graph. Our results are tight (up to a constant) when the $G$-index and dimension of $M$ coincide.  
	\end{abstract}
	
	\section{Introduction}
	
	Given $\epsilon>0$ and $d\geq 1$, the Borsuk graph $\text{Bor}^d(\epsilon)$ is the infinite graph with vertex set the unit $d$-sphere $\S^d$ and edges $x\sim y$ whenever $\dist[\S^d]{x,y}\geq\pi-\epsilon$, that is, if the two points are $\epsilon$-almost antipodal. It is well known that when $\epsilon$ is sufficiently small, its chromatic number is $d+2$, which follows from the topology of $\S^d$ via Borsuk-Ulam’s Theorem, and an explicit $(d+2)$ coloring given by projecting the facets of an inscribed regular $(d+1)$-simplex (see e.g. \cite{Matousek2008,Kahle-Martinez2020}). 
	
	The Borsuk graph has been studied in relation to Borsuk's conjecture and distance graphs (see e.g. \cite{Raigorodskii2012, Prosanov-Raigorodskii-Sagdeev2017,Sagdeev18,Barg2014}). It is also mentioned in Lov\'asz's proof of Kneser's conjecture \cite{Lovasz1978} where he initiated the study of topological obstructions to the chromatic number of graphs. This study has been continued, among others, by Babson's and Kozlov's work on the topology of Hom-complexes of graphs \cite{BabsonKozlov2003,BabsonKozlov2006,Kozlov2008}. In \cite{Kahle2007}, Kahle computed that these topological lower bounds are not efficient for the chromatic number of Erd\H{o}s--R\'enyi random graphs. In contrast, in \cite{Kahle-Martinez2020} we showed that the topology of the underlying $d$-sphere still dictates the chromatic number of random Borsuk graphs, and we described the threshold of $\epsilon$ for this to happen. 
	
	Regarding the antipodal map on a sphere as an action of the cyclic group of order two $\Z_2=\{\1,\nu\}$, we can equivalently define the Borsuk graph $\text{Bor}^d(\epsilon)$ as the graph with vertices $\S^d$, and edges $x\sim y$ whenever $\dist[\S^d]{x,\nu(y)}\leq\epsilon$.  This formulation naturally generalizes the Borsuk graphs to any metric space where a group $G$ acts. Thus, if $G$ is a group acting on a metric space $M$, and $\epsilon>0$ is given, the $G$-Borsuk graph on $M$, denoted $\Ggraph{G}{M,\epsilon}$, is the graph with vertex set $M$, and edges $x\sim y$ whenever there exist $g\in G, g\neq\1$ s.t.\ $\dist[M]{x,gy}\leq\epsilon$. In this paper, we study the chromatic number of such $G$-Borsuk graphs, when $G$ is a finite group and the underlying space $M$ is finitely triangulable.  
	
	Just as with regular Borsuk graphs, the chromatic number of $G$-Borsuk graphs is determined by the topology of the underlying space, however accurately bounding these requires more work. To get lower bounds, we study the corresponding hom-complexes as $G$-spaces, applying standard tools for $G$-spaces as described in \cite[Ch. 6]{Matousek2008}. We also establish some upper bounds, however this are far away from the lower bounds.  Computer-aided examples suggest that the topological lower bounds might actually tight. 
	
	We also generalize the random Borsuk graphs studied in \cite{Kahle-Martinez2020}: a random $G$-Borsuk graph on $M$, is the induced subgraph by taking $n$ uniform i.i.d. random points on $M$. For these, we exhibit threshold for $\epsilon$ such that asymptotically almost surely the chromatic number of the random subgraph coincides with that of the whole $G$-Borsuk graph.
	
	This paper is organized as follows. In section 2, we summarize the necessary background on $G$-spaces and $G$-simplicial complexes needed for the following sections. We also recall the definition of graph homomorphisms, and describe some conventions that we use throughout this paper. In section 3, we summarize our definitions and theorems about $G$-Borsuk graphs. In section 4, we prove the topological lower bound for the chromatic number. In section 5, we prove our upper bound. In section 6, we study the random $G$-Borsuk graphs. Finally, we include an appendix where we describe our approach to produce computer-aided examples.

	\section{Background}
	
	In this section we recall the standard definitions and properties of $G$-spaces and $G$-simplicial complexes, including $E_dG$ spaces and the $G$-index. Our exposition follows very closely that of Matou\v{s}ek in \cite[Ch. 6]{Matousek2008}. We also include the definition of graph homomorphisms and their relation to graph colorings (see more at \cite[Ch. 1]{Hell-Nesetril2004}). In addition, we make some useful conventions for the remainder of this paper. 
	
	\begin{defi}[$G$-spaces]
		Given a finite group $G$, a topological space $M$, is a \textbf{$\bm{G}$-space} if $G$ acts on $M$ via continuous maps. We say it is a \textbf{free $\bm{G}$-space} if the $G$-action is \textbf{free}, that is, if for all $g\in G$, $g\neq\1$ the map $g:M\to M$ has no fixed points. 
	\end{defi}
	
	\begin{defi}[Simplicial $G$-spaces] Let $G$ be a finite group.
		\begin{itemize}
			\item A simplicial complex $K$ is a \textbf{$\bm{G}$-simplicial complex} if $G$ acts on $K$ via simplicial maps. We say the action is \textbf{free}, if for every face $\sigma\in K$ and every $g\in G$ with $g\neq\1$, we have $\sigma\cap g\sigma=\emptyset$. 
			\item Moreover, a geometric realization of $K$, is a \textbf{geometric $\bm{G}$-simplicial complex}, if $G$ acts on it via linear extensions of simplicial maps. In this case, we consider $\|K\|$ as a metric space with the shortest path distance. Note that the $G$-action is indeed continuous with this metric. (Here, and throughout this paper, $\|\cdot\|$ denotes a fixed geometric realization). 
		\end{itemize}
	\end{defi} 
	
	\begin{center}
		\textbf{In this paper: we will only consider \textbf{free} $\bm{G}$-spaces and \textbf{free} $\bm{G}$-simplicial complexes, even if we don't say this explicitly.} 
	\end{center}	
	Note that $K$ being a free $G$-simplicial complex is a stronger condition than just asking it to be free on the geometric realization. However, if $K$ is a compact triangulable space and $G$ acts on it freely, we can always find a small enough triangulation such that the simplicial approximation theorem produces a free $G$-simplicial complex.
	
	We will mostly focus on compact metric $G$-spaces where the actions are via Lipschitz maps or isometries. Note that when $K$ is a geometric $G$-simplicial complex,  $G$ acts on it via piece-wise linear maps. Thus, if $K$ has a finite number of faces, the $G$-action is automatically Lipschitz. 
	
	\begin{defi}[$G$-maps]
		Let $G$ be a finite group. 
		\begin{enumerate}
			\item Let $X$ and $Y$ be $G$-spaces. A map $\phi:X\to Y$ is a \textbf{$\bm{G}$-map} if $\phi(g(x))=g(\phi(x))$ for all $x\in X$, and all $g\in G$. 
			\item Let $K$ and $M$ be $G$-simplicial complexes. A simplicial map $\phi:K\to M$ is a $G$-map if $\phi(gv)=g\phi(v)$ for all $v\in V(K)$ and all $g\in G$.  
		\end{enumerate} 
	\end{defi}
	
	\noindent\textbf{Notation:} In both cases we denote that $\phi$ is a $G$-map by writing $\phi:X\xrightarrow{G}Y$. We will also write just $X\xrightarrow{G} Y$, to mean that there exists some $G$-map from $X$ to $Y$. It is straightforward to check that composition of $G$-maps produces a $G$-map and that the identity is trivially a $G$-map. More generally, $G$-spaces with $G$-maps form a category, and (geometric) $G$-simplicial complexes with $G$-maps form a category as well.\\
	
	When $G=\Z_2$, the canonical example of a $\Z_2$-space is the $d$-sphere $\S^d$ with the antipodal map. For examples of $\Z_2$-simplicial complexes, we can take any triangulation of $\S^d$ that is symmetrical with respect to the origin. In particular, the cross-polytopes $\diamond^d$ provide a family of such examples. Having one canonical $\Z_2$-space for each dimension, makes the spheres the perfect benchmark for comparing with other $\Z_2$-spaces. The idea of \textit{benchmark spaces} can be extended to other groups via the finite classifying spaces $E_dG$, here again we follow the notation of \cite{Matousek2008}. 
	
	\begin{defi}[Finite Classifying Spaces]\label{defi:E_d_G}
		Let $G$ be a finite group, and $d\geq0$. A \textbf{finite classifying space of dimension} $\bm{d}$, denoted $E_dG$, is 
		\begin{enumerate}
			\item a finite free $G$-simplicial complex,
			\item $d$-dimensional, and
			\item $(d-1)$-connected
		\end{enumerate}
	\end{defi}
	
	This definition doesn't rule out the existence of different $E_dG$ classifying spaces for the same group $G$ and dimension $d$. However, they are always \textit{$G$-equivalent} in the sense that if both $X$ and $Y$ are $E_dG$ spaces, then we always have $\|X\|\xrightarrow{G}Y$ and $\|Y\|\xrightarrow{G}X$ \cite[Lemma 6.2.2 ]{Matousek2008}. Let us give some useful examples of $E_dG$ spaces. 
	
	\begin{examples}\label{ex:EdG_spaces}\leavevmode
		\begin{enumerate}
			\item For $G=\Z_2$, any symmetrical triangulation of $\S^d$ is a $E_d\Z_2$ space. In particular, the $d$-dimensional cross-polytope $\diamond^d$, which can be described via simplicial joins as $$\diamond^d=(\S^0)^{*(d+1)}=\underbrace{\S^0*\cdots*\S^0}_{d+1\text{ times}}.$$ 
			\item By some abuse of notation, we can regard a finite group $G$ as a $0$-dimensional simplicial complex with its vertices labeled by the elements of the group. Then, for any $d\geq 0$, $G$ acts freely on the space $G^{*(d+1)}$ via left-multiplication on each coordinate of the simplicial join. By elementary properties of the simplicial join operation, $G^{*(d+1)}$ is $(d-1)$-connected and $d$-dimensional, so it is an $E_dG$ classifying space. Also, by the symmetries of this construction, it can be seen that it can be realized by a geometric simplicial complex where $G$ acts via isometries. 
			\item Every cyclic group $\Z_m$ acts freely on $\S^1$ via $(2\pi)/m$ rotations. Since $\S^1$ is 1-dimensional and connected, it gives a $E_1\Z_m$ space. 
			\item Following the previous item, each cyclic group $\Z_m$ acts freely on the space $\Z_m*\S^1$, which is simply-connected and 2-dimensional, so it is a $E_d\Z_m$ space. We can understand this space as a collection of $m$ cones over the same $\S^1$. Pictorially, we represent these spaces by showing each $v*\S^1$ as an independent disk, where the boundaries must be identified. We use this depiction in Figure \ref{fig:upper_bounds}, where we draw the disks at the vertices of a regular $m$-gon, so that the $\Z_m$ action can be observed by rotating the picture by multiples of $2\pi/m$. 
		\end{enumerate}
	\end{examples}
	
	Comparing any $G$-space with the finite classifying spaces, the $G$-index is defined as follows. 
	
	\begin{defi}[$G$-index]
		For a $G$-space $M$, its \textbf{$\bm{G}$-index} is defined as		
		$$\ind{G}{M}:=\min\{d: M \xrightarrow[]{G} E_dG\}$$
	\end{defi}
	
	We summarize some properties of the $G$-index. 
	
	\begin{theorem}[{\cite[Prop. 6.2.4]{Matousek2008}}]\label{thm:properties-G-index}\leavevmode
		\begin{enumerate}
			\item If $X\xrightarrow{G} Y$ then $\ind{G}{X}\leq\ind{G}{Y}$.
			\item $\ind{G}{E_dG}=d$ for all $d\geq 0$. 
			\item $\ind{G}{X*Y}\leq \ind{G}{X}+\ind{G}{Y}+1$.
			\item If $X$ is $k$-connected, then $\ind{G}{X}\geq k+1$. 
			\item If $K$ is a free simplicial $G$-complex of dimension $d$, then $\ind{G}{X}\leq d$. 
		\end{enumerate}
	\end{theorem}
	
	Note item 2 is a generalization of Borsuk--Ulam's Theorem, and guarantees the $G$-index is well defined.
	
	Throughout this paper whenever we say $H$ is a graph, we mean it is a finite, simple, loop-less graph. That is, its vertex set $V(H)$ is a finite set, and its edges $E(H)$ are unordered pairs of vertices. We will also use the notation $u\sim v$ between vertices of $H$, to represent the edge $\{u,v\}\in E(H)$. We will denote the complete graph on $m$ vertices by $K_m$. The next couple of definitions are meant as a quick review of graph homomorphisms and colorings, for a more detailed exposition see \cite[Ch. 1]{Hell-Nesetril2004}. 
	
	\begin{defi} Given two graphs $H_1$ and $H_2$, a \textbf{graph homomorphism} from $H_1$ to $H_2$ is a function $f:V(H_1)\to V(H_2)$ defined on their vertices such that whenever $u\sim v$ in $H_1$ then also $f(u)\sim f(v)$ in $H_2$. 
	\end{defi}
	
	We can thus understand graph homomorphisms as functions that send vertices to vertices and edges to edges. We will denote them as $f:H_1\to H_2$. Again, we may just write $H_1\to H_2$ to mean that there exists some graph homomorphism from $H_1$ to $H_2$. It is known that the collection of graphs with graph homomorphisms form a category (see \cite{Kozlov2008,Hell-Nesetril2004}). 
	
	Given two graphs $H_1$ and $H_2$, we will denote the set of graph homomorphisms from $H_1$ to $H_2$ as $\hom(H_1,H_2)$. Babson and Kozlov \cite{BabsonKozlov2003,BabsonKozlov2006} introduced a canonical way of endowing $\hom(H_1,H_2)$ with a topological structure. We will discuss more about it in section \ref{section:lower_bound}.
	
	Graph colorings and chromatic number can also be described via graph homomorphisms. 	
	\begin{defi}
		Given a graph $H$, an \textbf{m-coloring} is a graph homomorphism $\phi:H\to K_m$. The \textbf{chromatic number} of $H$, is the least $m$ such that an $m$-coloring exits, i.e. $$\chi(H):=\min\{m:\hom(H,K_m)\neq\emptyset\}.$$
	\end{defi}
	
	\section{Definitions and Results}
	
	We start by giving a definition that generalizes Borsuk graphs on spheres to any metric $G$-space. 
	
	\begin{defi}[$G$-Borsuk Graphs]
		Let $G$ be a finite group. For a metric $G$-space $M$, $X\subset M$, and $\epsilon>0$ we define the \textbf{$\bm{G}$-Borsuk Graph} $\Ggraph{G}{X,\epsilon}$ as the graph with vertices $V(\Ggraph{G}{X,\epsilon})=X$ and edges $x\sim y$ whenever there is a non-identity element $g\in G$ such that $\dist[M]{x,gy}\leq\epsilon$.
	\end{defi}
	
	Thus the $\Z_2$-Borsuk Graph $\Ggraph{\Z_2}{\S^d,\epsilon}$, where $\Z_2$ acts on $\S^d$ via the antipodal map, is \textit{the} Borsuk Graph on $\S^d$. It is well known that the chromatic number of the $\Z_2$-Borsuk Graph on the $d$-Sphere is $d+2$, which follows from Borsuk--Ulam's Theorem (see e.g. \cite{Matousek2008}). More specifically, it depends on Lyusternik--Schnirerlman's Theorem, which is equivalent to Borsuk--Ulam's, concerning the minimum number of open sets covering $\S^d$ s.t.\ they don't contain pairs of antipodal points. In this spirit, we define the $G$-covering number of a $G$-metric space, and study its relation with $G$-Borsuk graphs. 
	
	\begin{defi}[$G$-covering number]
		Let $G$ be a finite group. Let $M$ be a free $G$-space. Then a (open)\textbf{ $\bm{G}$-cover} of $M$ is an open cover $M=U_1\cup\dots\cup U_k$, such that $U_i\cap gU_i=\emptyset$ for all $i=1, \dots, k$ and all $g\in G, g\neq\1$. Then the (open) \textbf{$\bm{G}$-covering number} of $M$ is:%
		$$\cov{G}{M}=\min\{k:\text{ there exists a } G\text{-cover } M=U_1\cup\dots\cup U_k\} $$
		Similarly, we define closed $G$-covers and the \textbf{closed $\bm{G}$-covering number} of $M$, $\overline{\cov{G}{M}}$, using closed covers instead.
	\end{defi}
	
	\paragraph{Notation:} Theorem \ref{thm:properties_cov} will show that the open and closed covering numbers coincide for compact spaces, and that they are invariant under $G$-equivalence. Since all $E_dG$ spaces are $G$-equivalent, the $G$-covering number $\cov{G}{E_dG}$ is independent of the specific classifying space, so we will denote it simply as $\cov{G}{d}$. 
	
	The $G$-covering number the chromatic number of $G$-Borsuk graphs on the same space are related via the next result. 
	
	\begin{theorem}\label{thm:rel_graph_cov_intro} If $M$ is a compact metric space and $G$ acts via Lipschitz maps. Then there exist constants $D$ and $\epsilon_0$ such that for any $0<\epsilon<\epsilon_0$ and any $X\subset M$ a $(\epsilon/D)$-net of $M$, we have $$\chi\left(\Ggraph{G}{X,\epsilon}\right)=\cov{G}{M}.$$
	\end{theorem}
	
	This relationship is what motivates us to find bounds for the $G$-covering number, and, at the same time this duality is one of our main tools finding these bounds. 
	
	\begin{theorem}\label{thm:inequalieties_cov_zm}
		Let $G$ be a finite group and $K$ a geometric $G$-simplicial complex. Let $k=\ind{G}{K}$. Then
		$$k+|G|\leq \cov{G}{K}\leq k(|G|-1)+2.$$	
	\end{theorem}
	
	Thus we get linear upper and lower bounds for $\cov{G}{K}$, depending only on the size of $G$ and the $G$-index of $K$. In the case $G=\Z_2$ we get the tight result $\cov{G}{K}= \ind{\Z_2}{K}+2$, where the lower bound is a generalization of Lyusternik--Schnirelmann--Borsuk Theorem, and the upper bound is obtained by ``coloring'' the facets of a regular $(k+1)$-dimensional simplex inscribed in $\S^k$ (see e.g.\cite{Kahle-Martinez2020, Matousek2002}). 
	
	In general, we conjecture that the lower bound gives the true chromatic number. If $K$ has \mbox{$G$-index}~$k$, then $K\xrightarrow{G}E_kG$, and so by Theorem \ref{thm:properties_cov}, $\cov{G}{K}\leq\cov{G}{E_kG}=\cov{G}{k}$, so we only need to upper bound the $G$-covering number for $E_dG$ spaces. Thus, we state our conjecture as follows. 
	
	\begin{conj}\label{conj:cov_d}
		Let $G$ be a finite group. Then $\cov{G}{d}=|G|+d$. 
	\end{conj}
	
	We also state the following less general conjectures, which may hold even if Conjecture \ref{conj:cov_d} results to be false. 
	
	\begin{conj}\label{conj:smaller_conjectures}\leavevmode
		\begin{enumerate}
			\item For $m\geq 2$ and $d\geq 0$, $$\cov{\Z_m}{d}=m+d.$$
			\item Let $G$ be a finite group. Then $\cov{G}{d}<\cov{G}{d+1}$ for all $d\geq0$. 
		\end{enumerate}
	\end{conj}
	
	Notice that item 2 in the conjecture asks if $\cov{G}{\cdot}$ is strictly increasing, since the fact that it is weakly monotone follows trivially from the $G$-map $E_dG=G^{*(d+1)}\xhookrightarrow{G}G^{*(d+2)}=E_{d+1}G$, and Theorem \ref{thm:properties_cov}.
	
	Table \ref{table:covering_numbers} summarizes the cases for which we have been able to compute $\cov{G}{d}$ exactly. Notice that for all of these, Conjecture \ref{conj:cov_d} holds. Specific $\Z_m$-coverings for $E_2\Z_m$ when $m=3,4,5$ and 6 are shown in Figure \ref{fig:upper_bounds}. Here we use $\Z_m*\S^1$ as $E_2\Z_m$ space, and is depicted as described in Examples~\ref{ex:EdG_spaces}. These $G$-covers were found solving integer linear programming problems with the free software \cite{sagemath} and a student license of \cite{gurobi}. We further explain our approach in the appendix.  
	
	\begin{table}
		\centering
		\begin{tabular}{|c|c||c|}
			\hline
			\textbf{\textbf{$\bm{G}$-Index}} & \textbf{\textbf{Group}} & \textbf{\textbf{Covering Number}} \\ 
			$\bm{d}$                         & $\bm{G}$                & \textbf{cov}$\bm{{}_{G}(d)}$      \\ \hline\hline
			any $d\geq0$                     & $\Z_2$                  & $d+2$                             \\ \hline
			0                                & any $G$                 & $|G|$                             \\ \hline
			1                                & any $G$                 & $|G|+1$                           \\ \hline
			\multirow{5}{*}{2}               & $\Z_3$                  & 5                                 \\ \cline{2-3} 
			&   $\Z_4$                      & 6                                 \\ \cline{2-3} 
			&   $\Z_2\times\Z_2$                & 6                                 \\ \cline{2-3} 
			&   $\Z_5$                & 7                                 \\ \cline{2-3} 
			&   $\Z_6$                 & 8                                 \\ \hline
			\multirow{1}{*}{3}               & $\Z_3$                  & 6                                 \\ 
			\hline
		\end{tabular}\caption{Known Covering Numbers}\label{table:covering_numbers}
	\end{table}
	
	\begin{figure}
		\centering
		\begin{subfigure}[b]{0.45\textwidth}
			\centering
			\includegraphics[width=\textwidth]{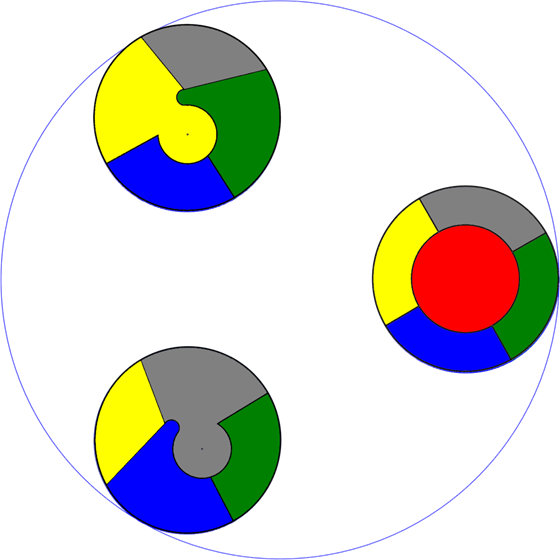}
			\caption{A cover of $E_2\Z_3$ with 5 colors.}
		\end{subfigure}
		\hfill
		\begin{subfigure}[b]{0.45\textwidth}
			\centering
			\includegraphics[width=\textwidth]{{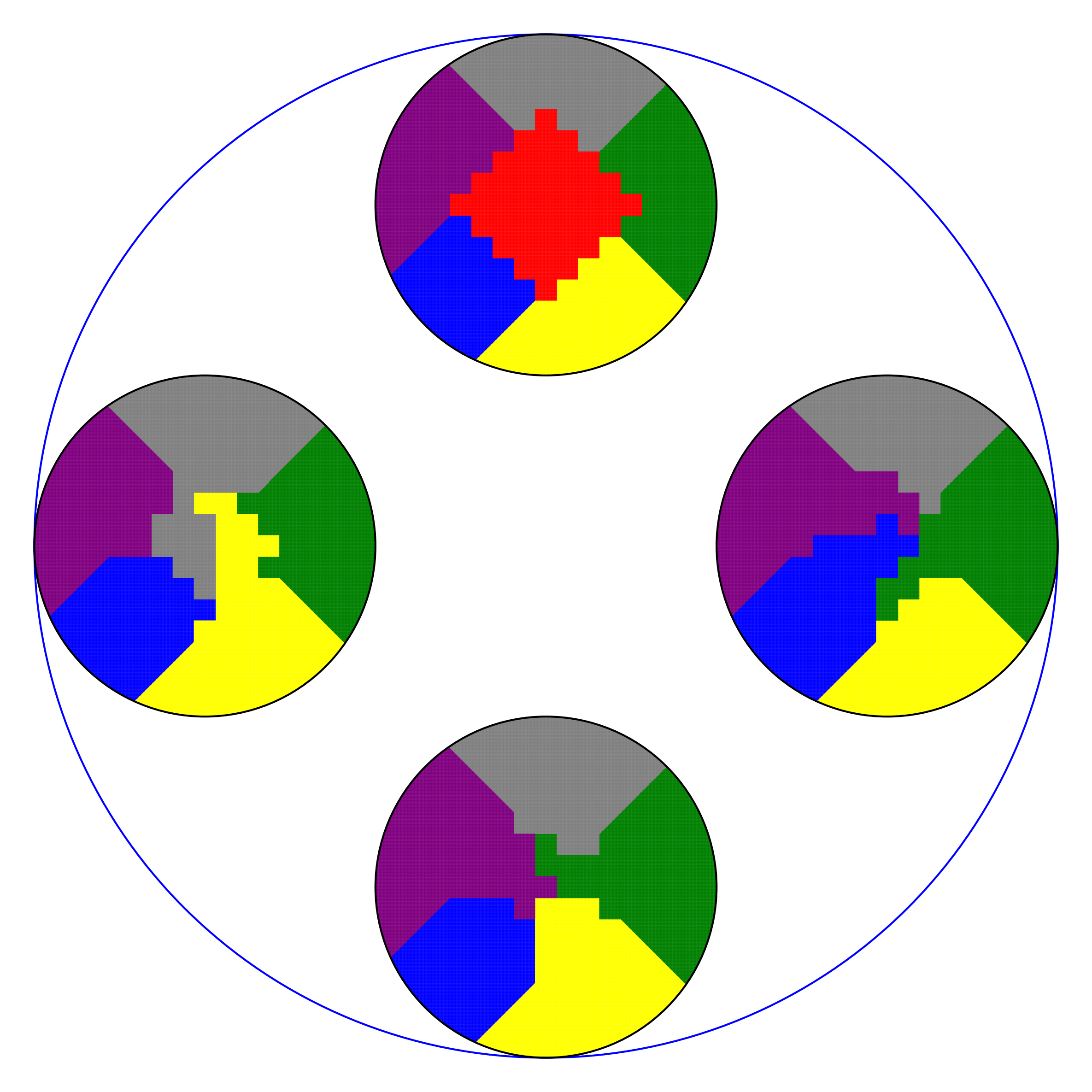}}
			\caption{A cover of $E_2\Z_4$ with 6 colors.}
		\end{subfigure}
		\hfill
		\begin{subfigure}[b]{0.45\textwidth}
			\centering
			\includegraphics[width=\textwidth]{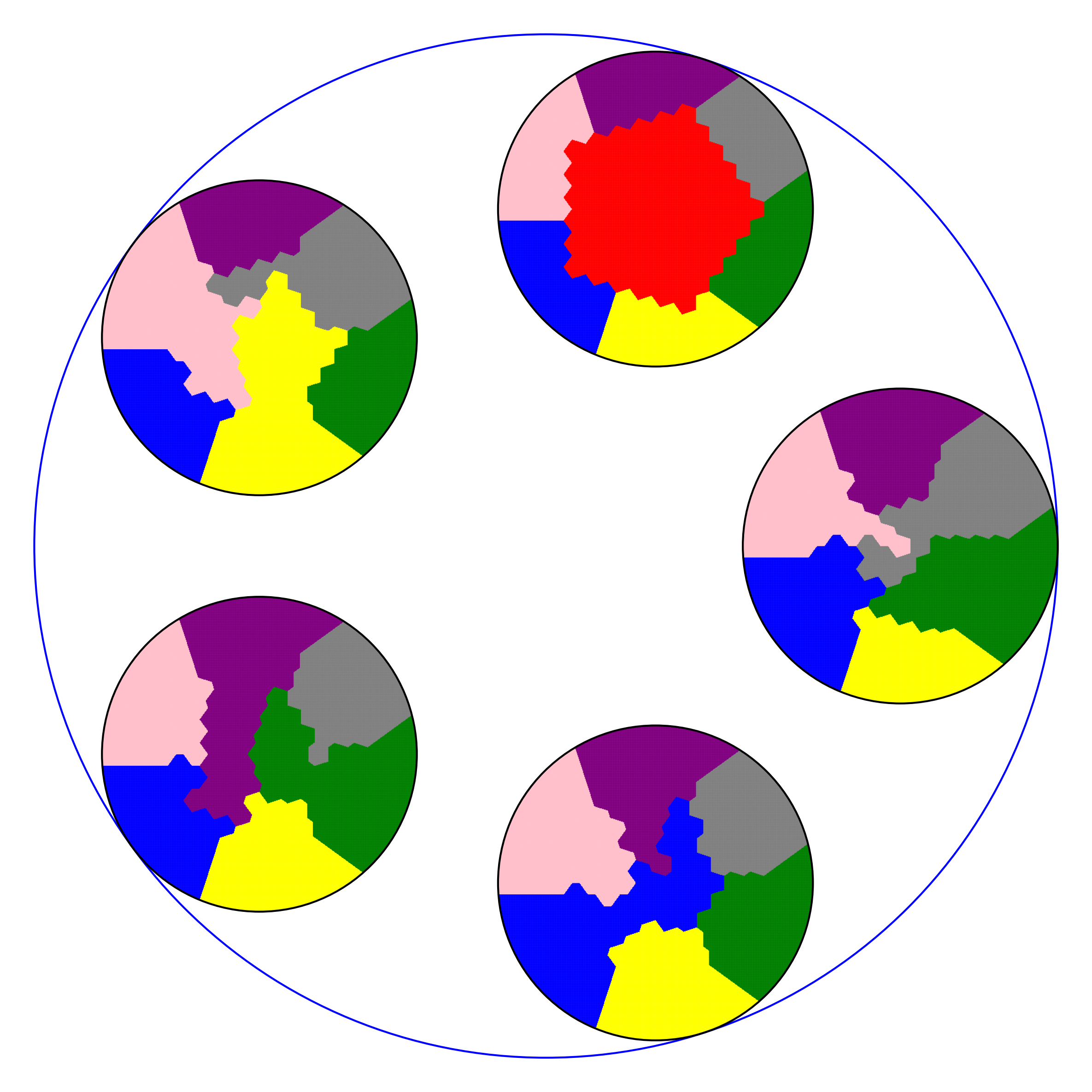}
			\caption{A cover of $E_2\Z_5$ with 7 colors. }
		\end{subfigure}
		\begin{subfigure}[b]{0.45\textwidth}
			\centering
			\includegraphics[width=\textwidth]{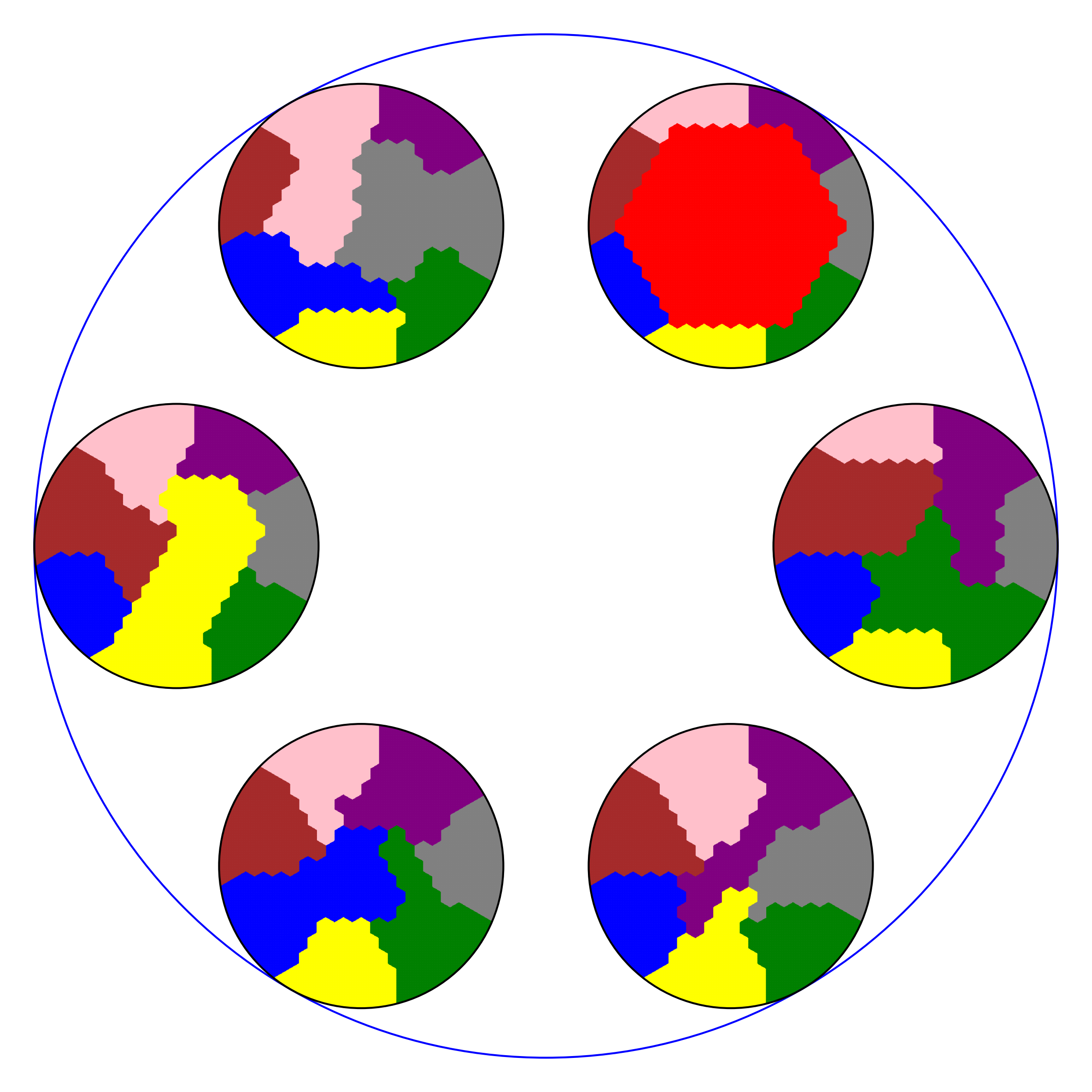}
			\caption{A cover of $E_2\Z_6$ with 8 colors. }
		\end{subfigure}
		\caption{Optimal $G$-coverings for 2-dimensional classifying spaces of $G=\Z_3,\Z_4, \Z_5$ and $\Z_6$}.
		\label{fig:upper_bounds}
	\end{figure}
	
	The lower bound in Theorem \ref{thm:inequalieties_cov_zm}, provides a generalization of the Lyusternik--Schnirelmann--Borsuk Theorem to $G$-spaces. 
	
	\begin{coro}[Generalization of LSB-Theorem]
		Let $G$ be a finite group and $K$ a finite $G$-simplicial complex. Suppose $M$ is $d$-connected and $|G|=m$. If $M=U_1\cup\dots\cup U_{d+m}$ is an open (or closed) cover of $K$, then there exist $x\in K$, an index $1\leq i\leq d+m$ and an element $g\in G\setminus\{\1\}$ such that $x,g(x)\in U_i$.  
	\end{coro}
	\begin{proof}
		By Theorem \ref{thm:properties-G-index}, $\ind{G}{K}\geq d+1$, and by Theorem \ref{thm:inequalieties_cov_zm}, $\cov{G}{K}\geq\ind{G}{K}+|G|\geq m+d+1$. Thus there are not open (nor closed, by Theorem \ref{thm:properties_cov}) $G$-covers of $K$. 
	\end{proof}
	
	We also generalize the random Borsuk graphs studied in \cite{Kahle-Martinez2020} to random induced subgraphs of $G$-Borsuk graphs. 
	
	\begin{defi}[Random $G$-Borsuk graphs]
		Let $G$ be a finite group, and $M$ a compact metric $G$-space. We define the \textbf{random $G$-Graph on $M$}, denoted by $\Ggraph{G}{n,\epsilon}$, as the $G$-Borsuk graph $\Ggraph{G}{X(n),\epsilon}$, where $X(n)=\{x_1,\dots, x_n\}$ is a set of $n$ i.i.d. uniform random points on $M$.
	\end{defi}
	
	In the same spirit as Theorems 1.2 and 1.3 of \cite{Kahle-Martinez2020}, the next result provides a tight threshold (up to a constant) for $\epsilon$, such that the random $G$-Borsuk graphs have the chromatic number equal to the $G$-covering number of its dimension. In this theorem, and throughout this paper, we will say that a random event $A_n$ happens \textit{asymptotically almost surely (a.a.s.)} if $\P{A_n}\to 1$ as $n\to\infty$. 
	
	\begin{theorem}\label{thm:chromatic_random_intro}
		Let $G$ be a finite group and $K$ a geometric $E_dG$ space. Then, there exist constants $C_1$ and $C_2$ (independent from $n$), such that:
		\begin{enumerate}
			\item If $\epsilon\geq C_1\left(\frac{\log n}{n}\right)^{1/d}$, then a.a.s. 
			$\chi(\Ggraph{G}{n,\epsilon})=\cov{G}{d}.$
			
			\item If $\epsilon\leq C_2\left(\frac{\log n}{n}\right)^{1/d}$, then a.a.s.
			$\chi(\Ggraph{G}{n,\epsilon})\leq\cov{G}{d-1}.$
		\end{enumerate}
	\end{theorem}
	
	Instead of proving Theorem \ref{thm:chromatic_random_intro}, we will prove the more general Theorem \ref{thm:chromatic_random}, which can be applied to any geometric $G$-simplicial complex, even when the $G$-index and dimension don't coincide. However, in such general cases, the thresholds found don't coincide. 
	
	A nice application of $G$-Borsuk graphs, is that they can produce graphs with large chromatic number and prescribed clique number. A similar construction using $G$-actions can also be found in \cite{Danesh2018}.
	
	\begin{coro}
		Given $m\leq t$ integers, there is a graph $H$ such that $\omega(H)=m$ and $\chi(H)\geq t$. 
	\end{coro}
	
	\begin{proof}
		Let $G$ be any group of order $m$, and let $d=t-m$. Consider $K$ a geometric $E_dG$ space. Let $\epsilon_1$ and $D_1$ be the corresponding constants given by Theorem \ref{thm:rel_graph_cov}, and $\epsilon_2$, $D_2$, the ones given by Lemma \ref{lemma:clique_number}. Let $\epsilon_0=\min\{\epsilon_1,\epsilon_2\}$ and $D=\max\{D_1,D_2\}$. Take any $X\subset E_dG$, an $(\epsilon/D)$-net, for some $\epsilon<\epsilon_0$. Take $H=\Ggraph{G}{X,\epsilon}$. Then, by Lemma \ref{lemma:clique_number}, $\omega(H)=m$. And by Theorems \ref{thm:rel_graph_cov} and \ref{thm:inequalieties_cov_zm}, 
		$$\chi(H)=\cov{G}{d}\geq |G|+d=t.$$
	\end{proof}
	
	\section[Properties of $\text{cov}_G$]{Properties of $\bm{\text{cov}_G}$}
	
	In this section we state and prove the main properties of $\cov{G}{\cdot}$ and its relationship with the chromatic number of $G$-Borsuk graphs. The following Theorem summarizes a few of these properties. 
	
	\begin{theorem}\label{thm:properties_cov}
		Let $G$ be a finite group, and $M$ and $K$ be $G$-spaces. Then:
		\begin{enumerate}
			\item If $M\xrightarrow[]{G}K$, then $\cov{G}{M}\leq \cov{G}{K}$. In particular, if $M$ and $K$ are $G$-equivalent, $\cov{G}{M}=\cov{G}{K}$. 
			
			\item If $\phi:M\xrightarrow{G}K$ is a Lipschitz map with constant $\lambda$, then for any $X\subset M$, and any $\epsilon>0$, it induces a graph homomorphism $$\phi:\Ggraph{G}{X,\epsilon}\to\Ggraph{G}{\phi(X),\lambda\epsilon}.$$
			
			\item If $M$ is a compact metric space and $G$-acts via Lipschitz maps, then $\cov{G}{M}=\overline{\cov{G}{M}}.$
			So the covering number can be computed using either open or closed covers. 
			
			\item If $G'\leq G$ is a subgroup, then $\cov{G'}{M}\leq\cov{G}{M}$
			
		\end{enumerate}
	\end{theorem}
	
	\begin{proof}\leavevmode
		\begin{enumerate}
			\item Let $k=\cov{G}{K}$, we may assume $k<\infty$ otherwise there's nothing to prove. Thus there is an open $G$-cover $K=U_1\cup\dots\cup U_k$ such that $U_i\cap g (U_j)=\emptyset$ for all $i$ and all $g\in G\setminus\{\1\}$. Let $f:M\xrightarrow[]{G} K$, then $M=f^{-1}(U_1)\cup\dots\cup f^{-1}(U_k)$ gives an open cover for $M$. Moreover, for each $i$, and $g\in G\setminus\{\1\}$, we have $$f(\left( f^{-1}(U_i)\cap g(f^{-1}(U_i))\right)\subset U_i\cap g(U_i)=\emptyset$$ so the sets $f^{-1}(U_i)$ are an open $G$-cover for $M$. Therefore, $\cov{G}{M}\leq k=\cov{G}{K}$. 
			
			\item We just need to check that $\phi$ \textit{maps edges to edges}. Take $x\sim y$ in $\Ggraph{G}{X,\epsilon}$, so $x,y\in X$ and there is some $g\in G, g\neq\1$ s.t.\ $\dist[M]{x,gy}\leq\epsilon$. Since $\phi$ is Lipschitz with constant $\lambda$, then $$\dist[K]{\phi(x),g\phi(y)}=\dist[K]{x,\phi(gy)}\leq\lambda\dist[M]{x,gy}\leq\lambda\epsilon,$$ and so $\phi(x)\sim\phi(y)$ in $\Ggraph{G}{\phi(X),\lambda\epsilon}$. Thus $\phi$ induces a graph homomorphism.  
			
			\item We may assume both $\cov{G}{M}$ and $\overline{\cov{G}{M}}$ are finite. If one of them is finite, the following proof gives that the other must be finite as well. Otherwise, they are both infinite.  
			
			\begin{itemize}
				\item Let $k=\overline{\cov{G}{M}}$. Write $M=F_1\cup\dots\cup F_k$, where the $F_i$'s are a closed $G$-cover. Note $F_i\times F_i$ is compact, and for each $g\neq\1$, the map $\dist{\cdot,g(\cdot)}: F_i\times F_i \to \R^+$ is continuous (it is indeed positive, because $F_i\cap g(F_i)=\emptyset$). Since $b_{i,g}:=\min\{\dist{x,g(y)}: x,y\in F_i\}$ is attained, it must be that $b_{i,g}>0$. Let $b=\min\{b_{i,g}: 1\leq i\leq k, g\in G, g\neq\1\}$, note $b>0$. 
				
				Since $G$ acts via Lipschitz maps, let $\lambda_g$ be the corresponding Lipschitz constant for the map $g\in G\setminus\{\1\}$. Let $D=1+\max\{\lambda_g: g\in G, g\neq\1\}$. For each $i=1,\dots,k$ define $U_i=\bigcup_{x\in F_i} \text{int}\left(B(x,b/D)\right)$. Thus the $U_i$'s are open, and since $F_i\subset U_i$, they also cover $M$. Now we prove the $U_i$'s are a $G$-cover.
				
				Suppose by way of contradiction that $y\in U_i\cap g(U_i)$. So there is $y'\in U_i$ with $y=g(y')$. By definition of $U_i$, there are $x,x'\in F_i$ such that $\dist{x,y}<b/D$ and $\dist{x',y'}<b/D$. Then%
				\begin{align*}
					b\leq b_{i,g}\leq \dist{x,g(x')} &\leq \dist{x,y}+\dist{y,g(x')} = \dist{x,y}+\dist{g(y'),g(x')}\\
					&\leq \dist{x,y} + \lambda_g\dist{y',x'} < \frac{b}{D}+\lambda_g\frac{b}{D}=\frac{1+\lambda_g}{D}b\leq b
				\end{align*}		
				But $b<b$ is a contradiction, so $U_i\cap g(U_i)=\emptyset$ for all $i$, $g\neq\1$. Therefore, \mbox{$\cov{G}{M}\leq\overline{\cov{G}{M}}$}. 
				
				\item Let $k=\cov{G}{M}$. Write $M=U_1\cup\dots\cup U_k$, where the $U_i$'s are an open $G$-cover. For each $x\in M$, choose a superset $x\in U_i$, and let $F_x$ be a closed neighborhood of $x$ contained in $U_i$. The collection $\{\text{interior}(F_x):x\in M\}$ is an open cover of $M$, by compactness, there is a finite subcover $\mathcal{C}$.
				
				For each $i=1,\dots, k$, let $F_i=\bigcup F_x$ taking the union over the closed sets such that $\text{interior}(F_x)\in\mathcal{C}$ and $F_x\subset U_i$. Thus $F_1, \dots, F_k$ is a closed cover of $M$. Since $F_i\subset U_i$, trivially $F_i\cap g(F_i)=\emptyset$ for all $i$, $g\in G$, $g\neq\1$. Therefore $\overline{\cov{G}{M}}\leq \cov{G}{M}$.
			\end{itemize}
			\item Note any $G$-cover of $M$, will also be a $G'$-cover of $M$, so trivially $\cov{G'}{M}\leq\cov{G}{M}$.  \qedhere
		\end{enumerate}	
	\end{proof}
	
	In the remainder of this paper we will only work on compact $G$-spaces. Theorem \ref{thm:properties_cov} establishes that the open and closed $G$-covering numbers coincide for compact spaces, so we will always denote it simply by $\cov{G}{M}$, even though most of the times we will produce \textbf{closed} $G$-covers. 
	
	The next Theorem explains the relation between the $G$-covering of a space and the chromatic number of a $G$-Borsuk graph on the same space. We choose to state this relation via two separate inequalities, since for some of the later results, we will only require one or the other. Theorem~\ref{thm:rel_graph_cov_intro} will then follow immediately. 
	\begin{theorem}\label{thm:rel_graph_cov}
		Let $G$ be a finite group, and $M$ a compact metric $G$-space s.t.\ $G$ acts on $M$ via Lipschitz maps. Then the following hold. 
		\begin{enumerate} 
			\item There exists a constant $D$ s.t.\ if $X\subset M$ is an $(\epsilon/D)$-net, we have $$\cov{G}{M}\leq\chi(\Ggraph{G}{X,\epsilon}).$$
			
			\item There exists a constant $\epsilon_0>0$ s.t.\ for any $X\subset M$, and any $0<\epsilon<\epsilon_0$ we have $$\chi(\Ggraph{G}{X,\epsilon})\leq\cov{G}{M}.$$
		\end{enumerate}
	\end{theorem}
	
	\begin{proof}
		Let $k=\cov{G}{M}$, so we can write $M=F_1\cup\dots\cup F_k$, where the $F_i$'s are a closed $G$-cover. Using the same notation as in the proof of Theorem \ref{thm:properties_cov} item 3, there exists $b,D>0$ such that $$b\leq \dist{x,g(y)}\; \forall x,y\in F_i,\; i=1,\dots, k,\; g\in G\setminus\{\1\},\text{ and}$$
		$$\dist{g(x),g(y)}\leq (D-1)\dist{x,y}\; \forall x,y\in M,\; g\in G\setminus\{\1\}.$$
		
		\begin{enumerate}
			\item Let $X\subset M$ be a $(\epsilon/D)$-net of $M$. Suppose by way of contradiction that there is a proper $(k-1)$-coloring for $\Ggraph{G}{X,\epsilon}$, and fix such a coloring. For each color $i=1,\dots, k-1$ define%
			$$U_i=\bigcup_{\substack{v\in X \\ v \text{ has color }i}}\text{int}\left(B\left(v,\frac{\epsilon}{D}\right)\right).$$
			Since $X$ is an $(\epsilon/D)$-net of $M$, the $U_i$'s are an open cover of $M$. Suppose there is an index $i$ and a non-identity $g\in G$ such that $x\in U_i\cap g(U_i)$. Thus $x=g(x')$ for some $x'\in U_i$. By definition, there are vertices $v,v'\in X$ of color $i$ such that $\dist{x,v}\leq\epsilon/D$ and $\dist{x',v'}\leq\epsilon/D$. But then%
			\begin{align*}
				\dist{v,g(v')} &\leq \dist{v,x}+\dist{x,g(v')}=\dist{v,x}+\dist{g(x'),g(v')}\\
				&\leq \dist{v,x}+ (D-1)\dist{x',v'}\leq \frac{\epsilon}{D}+(D-1)\frac{\epsilon}{D} =\epsilon
			\end{align*}
			So $v\sim v'$ are adjacent in $\Ggraph{G}{X,\epsilon}$, and they both have the same color $i$, which is absurd. Hence $U_i\cap g(U_i)=\emptyset$ for all $i$ and all $g\in G\setminus\{\1\}$. But this is a $(k-1)$ open cover of $M$, which is a contradiction since $\cov{G}{M}=k$. Therefore $k\leq\chi(\Ggraph{G}{X,\epsilon})$.
			
			\item Define $\epsilon_0=b$, and let $0<\epsilon<\epsilon_0$. Let $X\subset M$. We can $k$-color the graph $\Ggraph{G}{X,\epsilon}$ via the coloring $c:X\to\{1,\dots, k\}$ given by $c(v)=\min\{i: v\in F_i\}$. Indeed, if $v,w\in X$ have $c(v)=c(w)=i$, then $v,w\in F_i$, so $\dist{v,g(w)}\geq b>\epsilon$ for all $g\in G\setminus\{\1\}$. Hence $v\not\sim w$, and $c$ is a proper $k$-coloring. Therefore $\chi(\Ggraph{G}{X,\epsilon})\leq k$.\qedhere
		\end{enumerate}		
	\end{proof}
	
	We finish this section on general properties, by computing the clique number of $G$-Borsuk graphs. In a sense, this can be seen as a generalization of Lemma 2.2 in \cite{Kahle-Martinez2020}, regarding the odd-girth of $\Z_2$-Borsuk graphs. 
	
	\begin{lemma}\label{lemma:clique_number}
		Let $M$ be a compact $G$-space such that $G$ acts via Lipschitz maps. Then, there exist constants $\epsilon_0$ and $D$ such that, whenever $X\subset M$ is a $(\epsilon/D)$-net of $X$ and $\epsilon<\epsilon_0$, then $\omega\left(\Ggraph{G}{X,\epsilon}\right)=|G|$. 
	\end{lemma}
	
	\begin{proof}
		Let $b=\min\{\dist{x,g(x)}: x\in M, g\in G, g\neq\1\}$. Note $b>0$ since the $G$-action is free and $M$ is compact. Let $m=|G|$ and enumerate $G=\{g_1=\1, g_2, \dots, g_m\}$. Let $\lambda$ be the largest Lipschitz constant of the maps $g_i$. Then set $D=2\max\{\lambda,\lambda^2\}+2$ and $\epsilon_0=b/D$. Suppose now that $X\subset M$ is an $(\epsilon/D)$-net with $0<\epsilon<\epsilon_0$. Let $z\in M$ be any point, thus for each $i=1,\dots, m$, there is a point $x_i\in X$ such that $\dist{x_i,g_iz}\leq \epsilon/D$. Then, for any $i\neq j$ we have%
		\begin{align*}
			\dist{x_i, (g_ig_j^{-1})x_j}& \leq \lambda\dist{g_i^{-1}x_i,g_j^{-1}x_j}\leq \lambda\left(\dist{g_i^{-1}x_i,z}+\dist{g_j^{-1}x_j,z}\right)\\
			&\leq \lambda^2\left(\dist{x_i,g_ix_i}+\dist{x_j,g_jz}\right)\leq \frac{\lambda^2\epsilon}{D}<\epsilon.			
		\end{align*}
		Since $g_ig_j^{-1}\neq\1$ whenever $i\neq j$, $x_i\sim x_j$, and so $\{x_1,\dots, x_m\}$ is a $m$-clique in $\Ggraph{G}{X,\epsilon}$.
		
		On the other hand, suppose by way of contradiction that $x_0, \dots, x_m\in X$ is a $(m+1)$-clique. So, in particular,  for each $1\leq i\leq m$, there is an index $k_i\neq 1$ s.t. $\dist{x_0,g_{k_i}x_i}\leq \epsilon$. Hence there must be two different values $i\neq j$, such that $g_{k_i}=g_{k_j}$; fix these $i$ and $j$. Since $x_i\sim x_j$, there must also be an element $g\in G, g\neq\1$ such that $\dist{x_j, gx_i}\leq\epsilon$.  Then, setting $y=g_{k_i}^{-1}x_0=g_{k_j}^{-1}x_0$, we get%
		\begin{align*}
			\dist{x_i,gx_i} &\leq \dist{x_i,y}+\dist{x_j,y}+\dist{x_j,gx_i}\\
			&=\dist{x_i,g_{k_i}^{-1}x_0}+\dist{x_j,g_{k_j}^{-1}x_0}+\dist{x_j+gx_i}\\
			&\leq \lambda\dist{g_{k_i}x_i,x_0}+\lambda\dist{g_{k_j}x_j,x_0}+\epsilon\leq (2\lambda+1)\epsilon\leq (2\lambda+1)\epsilon_0<b
		\end{align*}
		But this is a contradiction, so there are not $(m+1)$-cliques, and $\omega(\chi(\Ggraph{G}{X,\epsilon}))=m=|G|$. 
	\end{proof}

	\section{Lower Bound}\label{section:lower_bound}
	In this section we focus on proving the lower bound of Theorem \ref{thm:inequalieties_cov_zm}. This result is topological, by bounding the $G$-index of the cellular structure of $\hom(K_m,H)$, following ideas of \cite{Lovasz1978}, \cite{BabsonKozlov2003} and \cite{BabsonKozlov2006}. It is worth noting, that similar $G$-actions on Hom-Posets have been independently studied in \cite{Danesh2018}. We start by restating the result we will prove. 
	
	\begin{theorem}\label{thm:lower_bound_cov}
		Let $G$ be a finite group and $K$ a geometric $G$-simplicial complex. Let $k=\ind{G}{K}$. Then $k+|G|\leq\cov{G}{K}$. 
	\end{theorem}
	
	Given two graphs $H_1$ and $H_2$, the set of graph homomorphisms $\hom(H_1,H_2)$ can be endowed with a topology as a prodsimplicial complex, as introduced by Babson and Kozlov in \cite{BabsonKozlov2003}. Let us explain this structure. 
	
	\begin{defi}
		A cell complex is called \textbf{prodsimplicial} if all its cells are products of simplices. I.e. if all of its cells are of the form $\sigma=\sigma_1\times\dots\times\sigma_k$, where $\sigma_i$ is a $d_i$-simplex. \textbf{Notation} instead of using the $``\times''$ symbol, we will write $\sigma=(\sigma_1,\dots, \sigma_k)$. 
	\end{defi}
	
	\begin{defi}
		Given $H_1$ and $H_2$, $\hom(H_1,H_2)$ is a prodsimplicial complex with cells $\sigma$ \mbox{satisfying}: 
		\begin{enumerate}
			\item $\sigma=\prod_{x\in V(H_1)}\sigma_x$, where $\sigma_x\subset V(H_2)$. That is, $\sigma$ is the direct product of simplices on the vertices of $H_2$, indexed by the vertices of $H_1$, and
			\item If $\phi:V(H_1)\to V(H_2)$ is such that $\phi(x)\in\sigma_x$ for all $x\in V(H_2)$, then it extends to a graph homomorphism $\phi:H_1\to H_2$. 
		\end{enumerate}
	\end{defi}
	
	We highlight the following points:
	\begin{itemize}
		\item The cells of $\hom(H_1,H_2)$ can be though of as indexed by functions $\eta:V(H_1)\to\left(2^{V(H_2)}\setminus\emptyset\right)$, such that if $\{x,y\}\in E(H_1)$ then $\eta(x)\times\eta(y)\subset E(H_2)$. 
		\item The dimension of the cell $\sigma$ is given by $\sum_{x\in V(G)}\text{dim}(\sigma_x)$. Thus the vertices (0-cells) of $\hom(H_1,H_2)$ are precisely the graph homomorphisms $\eta:G\to H$. 
		\item The 1-skeleton can be seen as a graph, with vertices the graph homomorphisms and edges $\phi\sim\psi$ if and only if they differ at exactly one value. 
	\end{itemize}
	
	\begin{lemma}\label{lemma:dim_hom_km_kd}
		Let $K_m$ denote the complete graph with $m$ vertices. If $m\leq t$, then%
		$$\dim\left(\hom(K_m,K_t)\right)=t-m$$
	\end{lemma}
	
	\begin{proof}
		Let $\sigma$ be a cell of $\hom(K_m,K_t)$. Thus $\sigma=\prod_{i=1}^m\sigma_i=(\sigma_1,\dots, \sigma_m)$ where each $\sigma_i\subset\{1,\dots, t\}$. By definition, when $i\neq j$, for all possible choices $v_i\in\sigma_i$ and $v_j\in\sigma_j$, we must have $v_i\sim v_j$ in $K_t$. This is true iff $v_i\neq v_j$. Hence, $\sigma_i\cap \sigma_j=\emptyset$ for $i\neq j$. Then%
		$$\dim(\sigma)=\sum_{i=1}^m\dim(\sigma_i)=\sum_{i=1}^m(|\sigma_i|-1)=\left(\sum_{i=1}^m|\sigma_i|\right) -m =\left|\bigcup_{i=1}^m\sigma_i\right| -m \leq |\{1,\dots, t\}|-m=t-m.$$
		
		Moreover, the dimension is exactly $t-m$, since $\sigma=(1,2,3,\dots,m-1,\{m,m+1,\dots,t\})$ is a $(t-m)$-dimensional cell. 
	\end{proof}
	
	\subsection[Hom$(K_m,\cdot)$ as G-space]{$\bm{\hom(K_m,\cdot)}$ as $\bm{G}$-space.}
	Let $G$ be a finite group with $|G|=m$. Enumerate $G=\{g_1=\1,g_2,g_3,\dots,g_m\}$. By identifying the elements of $G$ with the vertices $\{1,\dots, m\}$ of the complete graph $K_m$, $G$ acts on $K_m$ via right multiplication. For a finite graph $H$, $G$ acts on $\hom(K_m,H)$ as follows:	If $g\in G$, and $\sigma=(\sigma_1,\dots, \sigma_m)$, then $g\sigma=(\sigma_{\pi(1)},\dots, \sigma_{\pi(m)})$, where $\pi(i)$ is the unique index such that $g_ig=g_{\pi(i)}$. Note that if $g\neq\1$, then $\pi(i)\neq i$ for all $i$. As we pointed out in the proof of Lemma \ref{lemma:dim_hom_km_kd}, if $\sigma$ is a cell in $\hom(K_m,H)$, it must be that $\sigma_i\cap\sigma_j=\emptyset$ for all $i\neq j$, in particular, 
	$$\sigma\cap g(\sigma)\subset(\sigma_{1}\cap \sigma_{\pi(1)},\dots,\sigma_m\cap\sigma_{\pi(m)})=\emptyset.$$
	
	So $G$ acts freely on $\hom(G,H)$ via cellular maps. Moreover, it is clear that this construction is functorial, so if $f:H\to H'$ is a graph homomorphism, then it induces a $G$-map $$f:\hom(K_m,H)\xrightarrow{G}\hom(K_m,H').$$
	
	\subsection{Proof of Theorem \ref{thm:lower_bound_cov}}
	Let $K$ be a free $G$-simplicial complex with $k=\ind{G}{K}$ and $m=|G|$. Fix $\epsilon>0$. Let $T$ be a finer triangulation of $K$, such that $T$ is also a $G$-simplicial complex and $\diam{\sigma}\leq\epsilon$ for all faces $\sigma\in T$. Let $H$ be the graph with vertex set $V(T)$ and edges $x\sim y$ whenever there exists $g\in G\setminus\{\1\}$ such that $\{x,g(y)\}\in T^{(1)}$. Note then $H$ is an induced subgraph of the $G$-Borsuk graph $\Ggraph{G}{V(T),\epsilon}$.
	
	Thus if $\epsilon<\epsilon_0$, for $\epsilon_0$ the constant given by Theorem \ref{thm:rel_graph_cov} item 2, we get 
	$$\chi(H)\leq\chi(\Ggraph{G}{V(T),\epsilon})\leq \cov{G}{K}.$$
	
	We will construct a $G$-map $\Phi:T\xrightarrow{G}\hom(K_m,H)$. Since $H\to K_{\chi(H)}$, the functoriality of the $G$-action on $\hom(K_m,\cdot)$ produces%
	$$ T\xrightarrow{G}\hom(K_m,H)\xrightarrow{G}\hom(K_m,K_{\chi(H)}).$$
	Finally, applying items 1 and 5 of Theorem \ref{thm:properties-G-index} and the dimension computed in Lemma \ref{lemma:dim_hom_km_kd}, we get
	\begin{align*}
		k&=\ind{G}{K}=\ind{G}{T}\leq\ind{G}{\hom(K_m,H)}\\
		&\leq\ind{G}{\hom(K_m,K_{\chi(H)}}\leq \dim(\hom(K_m,K_{\chi(H)}))\leq \chi(H)-m,
	\end{align*}
	so $k+m\leq\chi(H)\leq\cov{G}{K}$ as desired. From now on, we focus on defining the map $\Phi$. 
	
	Enumerate $G=\{g_1=\1,g_2,g_3,\dots, g_m\}$. Let $\sigma$ be a $d$-dimensional face of $K$, for any $0\leq d\leq\dim(K)$. Let $\Psi_\sigma=\prod_{g\in G} g(\sigma)=(\sigma,g_2(\sigma),g_3(\sigma),\dots, g_m(\sigma))$. Note $\Psi_\sigma$ is isomorphic to the contractible prodsimplicial complex $(\Delta_d)^m$, where $\Delta_d$ is the $d$-dimensional simplex, and the power $m$ refers to the Cartesian Product. 
	
	We claim $\Psi_\sigma$ is a subcomplex of $\hom(K_m,H)$. Indeed $g_i(\sigma)\subset V(H)$, so $\Psi_\sigma$ is the product of simplices in $V(H)$ indexed by the vertices of $K_m$. To prove our claim, we need that whenever $i\sim j$ in $K_m$, i.e. $i\neq j$, any choice $x\in g_i(\sigma)$ and $y\in g_j(\sigma)$ produces an edge $x\sim y$ in $H$. Pick any such $x$ and $y$, so $x=g_i(v)$ and $y=g_j(w)$ for some vertices $v,w\in \sigma$. Put $g=g_ig_j^{-1}$, then $$\{x,g(y)\}=\{g_i(v),g_i(w)\}=g_i(\{v,w\})$$ 
	Since $\{v,w\}\subset\sigma$ is an edge in $T$, so is $g_i(\{v,w\})$, and thus $x\sim y$ in $H$ as needed. 
	Moreover, note that for any $g\in G$,
	\begin{align*}
		g(\Psi_\sigma) 	&= g(g_1(\sigma),g_2(\sigma),\dots,g_m(\sigma))\\
		&=(g_{\pi(1)}(\sigma),g_{\pi(2)}(\sigma),\dots,g_{\pi(m)}(\sigma))\\
		&=(g_1g(\sigma),g_2g(\sigma),\dots,g_mg(\sigma))=\Psi_{g(\sigma)}.	
	\end{align*}
	Note also, if $\tau$ is a face of $\sigma$, $\Psi_\tau$ is a subcomplex of $\Psi_\sigma$, since for any $g\in G$, $g(\tau)\subset g(\sigma)$.\\
	
	Recall we say that a topological space $X$ is $n$-connected if for every $\ell=-1,0,\dots, n$, each continuous map $f:\S^\ell\to X$ can be extended to a continuous map $f:B^{\ell+1}\to X$. It is well known that any contractible space is $n$-connected for all $n\geq 0$. Recall the $d$-simplex $\Delta_d$ is homeomorphic to the closed ball $B^{d+1}$ and its boundary $\partial\Delta_d$ is homeomorphic to the $(d-1)$-sphere $\S^{d-1}$. Putting these together with the fact $\Psi_\sigma$ is contractible, we get the following tool:
	\begin{center}
		\textit{If $f:\partial \sigma \to \|\Psi_\sigma\|$ is a continuous map, then it can be extended into a map $f:\sigma\to\|\Psi_\sigma\|$. }
	\end{center}
	
	We now construct $\Phi$ by induction on the $d$-skeleton of $T$. 
	
	\begin{itemize}
		\item \textbf{Vertices: } For each vertex $v\in V(T)$, define $\Phi(v)=(v,g_2(v),\dots,g_m(v))=\Psi_{\{v\}}$.
		Clearly $\Phi(g(v))=\Psi_{\{g(v)\}}=g(\Psi_{\{v\}})=g(\Phi(v))$ as proved above. 
		\item \textbf{Induction Hypothesis: } Suppose that we have defined $\Phi: T^{(d)}\xrightarrow{G}\hom(K_m,H)$ for the $d$-skeleton of $T$, in such a way that $\Phi(\sigma)\subset\|\Psi_{\sigma}\|$. 
		\item \textbf{Inductive Step: } 
		Separate the $(d+1)$-cells of $T$ into $G$-orbits (this is possible because $G$ acts freely). Let $\mathcal{O}=G\sigma$ be one of such orbits, with $\sigma\in T^{(d+1)}$ a generator. Let $\tau_0,\tau_1,\dots,\tau_d$ be the $d$-faces of $\sigma$. By the induction hypothesis, $\Phi$ is defined such that $\Phi(\tau_i)\subset\|\Psi_{\tau_i}\|\subset \|\Psi_\sigma\|$. Thus $\Phi$ is defined as a map $$\Phi:\partial\sigma=\bigcup_{i=0}^d\tau_i\to\bigcup_{i=0}^d\|\Psi_{\tau_i}\|\subset\|\Psi_\sigma\|,$$ hence we can extend $\Phi$ to $\sigma$, such that $\Phi(\sigma)\subset\|\Psi_\sigma\|$. For each other $(d+1)$-cell $\omega\in\mathcal{O}$, there is a $g\in G$, such that $\omega=g(\sigma)$, so define%
		$$\Phi(\omega):=g(\Phi(\sigma))\subset g\left(\|\Psi_\sigma\|\right)\subset\|\Psi_{g\sigma}\|=\|\Psi_\omega\|.$$%
		Note this agrees on the $d$-faces of $\omega$, since we supposed $\Phi$ was a $G$-map on the $d$-skeleton of $T$. Finally, repeating this for all the $G$-orbits of $(d+1)$-cells, we get $\Phi:T^{(d+1)}\xrightarrow{G}\hom(K_m,H)$ satisfying the induction hypothesis. 
	\end{itemize}
	
	This produces the $G$-map $\Phi:T\xrightarrow{G}\hom(K_m,H)$ and finishes the proof. \hfill\qedsymbol
	\section{Upper Bounds}
	We now focus on the upper bound in Theorem \ref{thm:inequalieties_cov_zm}. Note that if $K$ is a $G$-space of $G$-index $k$, $K\xrightarrow{G}E_kG$, so by Theorem \ref{thm:properties_cov}, $\cov{G}{K}\leq\cov{G}{E_kG}=\cov{G}{k}$. Hence, we only need to prove an upper bound for finite $G$ classifying spaces. We do this by combining Theorem \ref{thm:recursive_upper_bound}, which provides a recursive upper bound, and Theorem \ref{thm:upper_bound_1dim}, which bounds $\cov{G}{1}$. We finish this section with Theorem \ref{thm:general_upper_bound}, which we use to compute upper bounds for $\cov{\Z_m}{d}$, for some values of $m$ and $d$, with the aid of a computer.   
	
	\begin{theorem}\label{thm:recursive_upper_bound}
		Let $G$ be a finite group and $d\geq 2$ an integer. Then $$\cov{G}{d}\leq\cov{G}{d-1}+|G|-1.$$
	\end{theorem}
	
	\begin{proof}
		Let $K=G*M$ where $M$ is a $(d-1)$-dimensional classifying space for $G$, so $K$ is a $d$-dimensional classifying space for $G$. The elements of $K$ can be described as linear combinations $tg\oplus (1-t)\mu$ where $g\in G, \mu\in M$ and $0\leq t\leq 1$, and thus $h\in G$ acts by
		$$h(tg\oplus (1-t)\mu)= hg\oplus (1-t)h(\mu).$$
		
		Let $k=\cov{G}{d-1}$, so there is a closed $G$-cover $M=F_1, \dots, F_k$. For each $i=1,\dots, k$, and each $g\in G, g\neq\1$, define the closed sets
		\begin{align*}
			F_{i,g}&=\left\{ tg\oplus (1-t)\mu\in K : \mu\in F_i, 0\leq t\leq 1/2 \right\}\text{ and }\\
			\bar{F}_{i,g}&=\left\{ tg \oplus (1-t)\mu\in K: \mu \in F_i, 1/2\leq t\leq 1 \right\}. 
		\end{align*}
		So $F_{i,g}\cup \bar{F}_{i,g}=g*F_i$. Also define the closed sets%
		\begin{align*}
			F_{i,\1}&=\left\{ t\1\oplus (1-t)\mu\in K: \mu \in F_i, 0\leq t\leq 1 \right\}=\1*F_i.
		\end{align*}
		Finally, let $m=|G|$ and enumerate $G=\{\1, g_1, g_2, \dots, g_{m-1}\}$. Define $E_1, \dots, E_k, E_{k+1}, \dots, E_{m+k-1}$ as follows
		\begin{align*}
			E_i &= \bigcup_{g\in G} F_{i,g} \text{ for } 1\leq i\leq k\text{ and}\\
			E_{k+i} &= \bigcup_{i=1}^k \bar{F}_{j,g_i}\text{ for }1\leq i\leq m-1. 
		\end{align*}
		Clearly the $(m+k-1)$ sets $E_i$ are closed and cover $K$. We claim they are a $G$-cover of $K$. 
		
		Suppose there is some index $1\leq i\leq k$ and some $h\in G, h\neq \1$ such that $E_i\cap hE_i\neq\emptyset$. Then, there must be some $g,g'\in G$ such that $x\in F_{i,g}\cap hF_{i,g'}\neq\emptyset$. Thus we can write $x$ as%
		$$x=tg\oplus (1-t)\mu=\tau hg'\oplus (1-\tau)h\nu,$$
		for some $\mu,\nu\in F_i$, $0\leq t,\tau\leq 1$. Then, it must be that $(1-t)\mu=(1-\tau)h\nu$. If $(1-t),(1-\tau)\neq 0$ we must have $\mu=h\nu$, but since $F_i\cap hF_i=\emptyset$, this is impossible. If instead $(1-t)=(1-\tau)=0$, then it has to be that $g=g'=\1$ and then $x=\1=h\1=h$, but $h\neq \1$ so this is impossible as well. 
		
		Suppose there is some index $1\leq i\leq m-1$, and some $h\in G, h\neq \1$ such that $E_{k+i}\cap hE_{k+i}\neq\emptyset$. Then, there must be some $1\leq j, \alpha\leq k$ such that $x\in \bar{F}_{j,g_i}\cap h\bar{F}_{\alpha,g_i}\neq\emptyset$. Thus we can write $x$ as%
		$$x=tg_i\oplus (1-t)\mu=\tau hg_i\oplus(1-\tau)h\nu,$$
		for some $\mu\in F_j$, $\nu\in F_\alpha$, $1/2\leq t, \tau\leq 1$. Then, it must be that $tg_i=\tau hg_i$, and since $t,\tau\neq 0$, we must have $g_i=hg_i$, but since $h\neq\1$, this is impossible. 
		
		Therefore $E_r\cap hE_r=\emptyset$ for all $1\leq r\leq k+m-1$ and all $h\in G, h\neq\1$ as desired. 	
	\end{proof}
	
	Now we state the upper bound for the 1-dimensional case. 
	\begin{theorem}\label{thm:upper_bound_1dim}
		Let $G$ a finite group. Then $$\cov{G}{1}\leq |G|+1.$$
	\end{theorem}
	
	Combining this Theorem with the lower bound in Theorem \ref{thm:lower_bound_cov}, gives $\cov{G}{1}=|G|+1$. To prove it, we need to first establish some lemmas. Let us start proving the case when $G$ is a cyclic group acting on the 1-dimensional classifying space $\S^1$. 
	
	\begin{lemma}\label{lemma:cover_circle}
		Consider $\S^1$ as a $\Z_m$ space, then $$\cov{\Z_m}{\S^1}\leq m+1.$$
	\end{lemma}
	
	\begin{proof}
		Note $\Z_m$ acts on $\S^1$ via rotations of $2\pi/m$. Enumerate $\Z_m=\{\1,\nu,\dots, \nu^{m-1}\}$. Note that for any $x\in \S^1$ and any $1\leq k\leq m-1$ we have%
		$$\dist[\S^1]{x,\nu^k(x)}\geq\dist[\S^1]{x,\nu(x)}=\frac{2\pi}{m}.$$ 
		For each $i=0,\dots, m$, let $F_i$ be the closed circular arc $[\frac{2i\pi}{m+1},\frac{2(i+1)\pi}{m+1}]$. Clearly $\diam{F_i}=\frac{2\pi}{m+1}<\frac{2\pi}{m}$, so it is impossible to have $x$ and $\nu^k(x)$ on the same $F_i$, that is $F_i\cap \nu^k F_i=\emptyset$ for all $i=0,\dots, m$ and all $k=1,\dots, m-1$. Therefore, the $F_i$'s are a $\Z_m$-cover of $\S^1$ using $(m+1)$ closed sets, and the result follows. 
	\end{proof}
	
	\begin{lemma}\label{lemma:cover_cyclic}
		Let $K=\Z_m*\Z_m$, so $K$ is trivially a $\Z_m$-space. As a simplicial complex, the vertices of $K$ are two disjoint copies of $\Z_m$. Denote these vertices by $1, \dots, m$ if they come from the first copy, and $\bar{1}, \dots, \bar{m}$, from the second. Then, there exists a closed $\Z_m$-cover $K=A_0\cup\dots\cup A_m$, with $(m+1)$ closed sets,  such that $i, \bar{i}\in A_i$ for each $i\in\Z_m=\{1,\dots, m\}$.  
	\end{lemma}
	
	\begin{proof}
		Let $\Phi:K\to\S^1$ be defined as follows. For the vertices of $K$, let $\Phi(i)=\Phi(\bar{i})=(2\pi i)/m\in\S^1$. For the edges $\{i,\bar{j}\}\in K$, $\Phi(\{i,\bar{j}\})=[\frac{2\pi i}{m}, \frac{2\pi j}{k}]$. This defines $\Phi$ on the simplicial complex $K$, and is straightforward to check it is both continuous and a $\Z_m$-map. 
		
		Lemma \ref{lemma:cover_circle} ensures there is a $\Z_m$-cover $F_0, \dots, F_m$ of $\S^1$. Clearly the points $\frac{2\pi i}{m}$ must be in different sets $F_i$ for it to be a $\Z_m$-cover, so we may assume $\frac{2\pi i}{m}\in F_i$ for $i=1,\dots, m$. Finally, the pullback cover $A_i=\Phi^{-1}(F_i)$ for $i=0,\dots, m$ produces the desired $\Z_m$-cover.  
	\end{proof}
	
	Regarding the closed sets of a $\Z_m$-cover as \textit{colors}, we can give an intuitive interpretation of Lemma \ref{lemma:cover_cyclic}. If the vertices of $K$ already have colors assigned, such that $c(i)=c(\bar{i})$ for all $i$, and $c(i)\neq c(j)$ whenever $i\neq j$, then we can extend such coloring to $K$, adding just one more color in order to get a $\Z_m$-cover. This interpretation is the key to prove Theorem \ref{thm:upper_bound_1dim}, where finding a $G$-cover for $G*G$ can be broken into finding $G$-covers for disjoint cyclic orbits of edges, using only the colors already assigned to the vertices and at most one more color.
	
	\begin{proof}[Proof of Theorem \ref{thm:upper_bound_1dim}]
		Let $m=|G|$ and enumerate $G=\{g_1=\1,g_2,\dots, g_{m}\}$. Let $K=G*G$, and denote its vertices by $g_1, \dots, g_m$ and $\bar{g_1},\dots, \bar{g_m}$, as in Lemma \ref{lemma:cover_cyclic}. All points $x\in \|K\|$ can be written as $tg_i\oplus (1-t)\bar{g_j}$, for some $g_i, g_j\in G$ and some $0\leq t\leq 1$. The edges of $K$ are all edges of the form $\{g_i,\bar{g_j}\}$, for $1\leq i,j\leq m$. To ease our notation, we will denote an edge $\{g_i, \bar{g_j}\}$ simply as $g_i\sim g_j$, where it is understood that the element on the left denotes that vertex and the element on the right denotes the \textit{overlined} version of the vertex. In particular, $g\sim g$ is valid, and denotes the edge $\{g,\bar{g}\}\in K$ for any $g\in G$. 
		
		We will construct a closed $G$-cover of $K$ by gradually adding closed subsets of $K$ to the initial color classes $C_0,\dots C_m$, initialized as $C_0=\emptyset$ and $C_i=\{g_i\}\cup\{\bar{g_i}\}$ for $1\leq i\leq m$. Note they are indeed closed, since they are either empty or exactly two points. 
		
		For an easier exposition, we will say that a closed subset $F\subset \|K\|$ is \textit{colored with color $g_i$, }to mean that we add the set $F$ to $C_i$, defining $C_i'=C_i\cup F_i$. This is a slight abuse of notation, since it is possible for some points to be part of many sets $C_i$. Note that by doing this finitely many times, the color classes will remain closed.
		
		We want that $C_i\cap gC_i=\emptyset$ for all $g\in G, g\neq\1$, $0\leq i\leq m$. This failing is equivalent to the existence of a point $x\in\|K\|$, an index $0\leq i\leq m$ and an element $g\in G$, $g\neq \1$ such that $x$ and $gx$ both \textit{have color} $g_i$. Since in our initial coloring the vertices of $K$ have different colors, $x$ must be in the interior of some edge $a\sim b$, and thus $gx$ is in the edge $ga\sim gb$. Thus, to prevent the existence of such $x$, we just need to focus on properly extending the color classes $C_0, \dots, C_m$, on the $G$-orbit $\calO(a\sim b):=\{ga\sim gb: g\in G\}$, for each edge $a\sim b\in K$ separately. When we say \textit{properly extending}, we mean that $C_0\cap\calO(a\sim b),\dots, C_m\cap\calO(a\sim b)$ is a closed $G$-cover of $\calO(a\sim b)$. We will now focus on doing this on one such orbit. 
		
		Let $a\sim b$ be an edge of $K$, and $\calO=\{ga\sim gb: g\in G\}$ be its $G$-orbit. By interpreting each edge as an assignment $\sigma(ga)=gb$, we can see $\sigma$ as a permutation of the elements of $G$. Recall that all finite permutations can be regarded as a disjoint union of cycles (see e.g. \cite[p.~30]{Lang2002}), so $\sigma$ partitions $G$ into $r$ cycles of the form
		\begin{align*}
			a_1^{(1)}\mapsto \sigma(a_1^{(1)})=a_2^{(1)}\mapsto \sigma(a_2^{(1)}) &=a_3^{(1)}\mapsto\cdots\mapsto\sigma(a_{k_1}^{(1)})=a_1^{(1)}\\
			a_1^{(2)}\mapsto \sigma(a_1^{(2)})=a_2^{(2)}\mapsto \sigma(a_2^{(2)}) &=a_3^{(2)}\mapsto\cdots\mapsto\sigma(a_{k_2}^{(2)})=a_1^{(2)}\\
			&\vdots\\
			a_1^{(r)}\mapsto \sigma(a_1^{(r)})=a_2^{(r)}\mapsto \sigma(a_2^{(r)}) &=a_3^{(r)}\mapsto\cdots\mapsto\sigma(a_{k_r}^{(r)})=a_1^{(r)}
		\end{align*}
		where $k_j$ is the length of the $j$-th cycle. 
		
		And at the same time, this partitions $\calO$ into $r$ disjoint subsets of the form $$\calO_j=\left\{ a_1^{(j)}\sim a_2^{(j)}, a_2^{(j)}\sim a_3^{(j)}, \dots, a_{k_j}^{(j)}\sim a_1^{(j)}     \right\}.$$
		
		Each $\calO_j$ can be seen itself as a 1-dimensional simplicial complex, and moreover it is naturally a $\Z_{k_j}$-space, where $\Z_{k_j}$ acts by cyclically permuting the edges. Thus, we can regard $\calO_j\subset \Z_{k_j}*\Z_{k_j}$. Lemma \ref{lemma:cover_cyclic} gives a coloring for the whole space $\Z_{k_j}*\Z_{k_j}$ just using the colors already assigned to the vertices $a_i^{(j)}$ and at most one more color (which we will assign to the color class $C_0$), so we may restrict this coloring to $\calO_j$. In the especial case when $\sigma$ is already a cycle, this finishes the proof. For the general case when $r\geq 2$, still Lemma \ref{lemma:cover_cyclic} will be our building block, but we need to be careful so that the different cycles don't produce points $x$ and $gx$ with the same color. 
		
		Take each edge $a_i^{(j)}\sim a_{i+1}^{(j)}$ (the addition $i+1$ must be understood modulo $k_j$) and divide it into $r$ closed intervals $P_{i,1}^{(j)}, P_{i,2}^{(j)},\dots, P_{i,r}^{(j)}$ intersecting only at their endpoints. More formally, we identify $a_i^{(j)}\sim a_{i+1}^{(j)}$ isometrically with the real segment $[0,1]$ by sending $a_i^{(j)}$ to 0 and $a_{i+1}^{(j)}$ to 1. Then $P_{k,i}^{(j)}$ is the preimage of the subinterval $\left[(k-1)/n,k/n\right]$. Notice then that for any $g\in G$, and any $1\leq k\leq r$ the set $gP_{i,k}^{(j)}$ is always another interval $P_{i',k}^{(j')}$, i.e. it is part of a different edge that may be part of a different orbit $\calO_{j'}$, but will always be on the same relative position over the edge. 
		
		We finally produce the coloring for $\calO=\{a_{i}^{(j)}\sim a_{i+1}^{(j)}: 1\leq i\leq k_j, 1\leq j\leq r\}$ by coloring each subinterval $P_{i,k}^{(j)}$, as follows.
		\begin{enumerate}
			\item Color each subinterval $P_{i,k}^{(j)}$ where $k<j$ with the color $a_{i}^{(j)}$. I.e. color the first $(j-1)$ subintervals of each edge, with the color of the left vertex. 
			\item Color each subinterval $P_{i,k}^{(j)}$ where $k>j$ with the color $a_{i+1}^{(j)}$. I.e. color the last $(r-j)$ subintervals of each edge, with the color of the right vertex. 
			\item Fix each $1\leq j\leq r$. Notice that the endpoints of the subintervals $P_{i,j}^{(j)}$, were already assigned colors $a_i^{(j)}$ and $a_{i+1}^{(j)}$ respectively. By cycling on the edges $a_i^{(j)}\sim a_{i+1}^{(j)}$, we also cycle on the subintervals $P_{i,j}^{(j)}$, so themselves can be seen as a subset of the $\Z_{k_j}$-complex $\Z_{k_j}*\Z_{k_j}$. Thus lemma \ref{lemma:cover_cyclic} gives a closed \textit{coloring} for $P_{1,j}^{(j)}, \dots, P_{k_j,j}^{(j)}$ using only the colors already assigned to the $a_i^{(j)}$'s and color 0 (color class $C_0$). 
		\end{enumerate}
		Checking that this coloring produces a closed $G$-cover of $\calO$, is straightforward. By repeating this process for all the $G$-orbits of edges in $K$, we get the desired $G$-cover with $m+1$ closed sets.
	\end{proof}
	
	A possible proof of Conjecture \ref{conj:cov_d}, would be to adapt the technique of the proof of Theorem \ref{thm:upper_bound_1dim} to higher dimensions. Doing this, would require to first prove Conjecture \ref{conj:smaller_conjectures} item 1, regarding $\Z_m$ spaces. While we haven't been able to do this generally, even for dimension 2, we have been able to explicitly find some $\Z_m$-covers for small values of $m$ aided by a computer. The following Theorem allows us to find $G$-covers computing the chromatic number of relatively small graphs. More details on the specific graphs and techniques used to produce the examples of Figure \ref{fig:upper_bounds} are included in the Appendix. 
	
	\begin{theorem}\label{thm:general_upper_bound}
		Let $K$ be a compact geometric $G$-simplicial complex. Let $T$ be a triangulation of $K$, such that it is still a $G$-simplicial complex. Define the graph $H$ to be the graph with vertices $V(T)$ and with edges $v\sim w$ whenever there exists a $g\in G, g\neq\1$ such that $\{v,gw\}\in T$. Then
		$$\cov{G}{K}\leq\chi(H).$$
	\end{theorem}
	\begin{proof}
		If $H$ has loops, its chromatic number is infinite and the result is trivial; so suppose $H$ has no loops. Let $c$ be a coloring of $H$ using $\chi(H)$ colors, so $c:V(T)\to\{1,\dots,\chi(H)\}$. 
		
		Let $B=\text{Bar}(T)$ be the Barycentric subdivision of $T$. Recall $\|B\|=\|T\|=\|K\|$ and that the simplices of $B$ can be identified with chains of simplices of $T$. Since $T$ is a $G$-simplicial complex, applying $g$ to a chain $\sigma_1\prec\sigma_2\prec\dots\sigma_r$ of simplices of $T$, produces $g\sigma_1\prec g\sigma_2\prec\dots\prec g\sigma_r$, another chain of simplices of $T$. So $B$ is a $G$-simplicial complex as well. Recall that the facets of $B$ correspond to maximal chains of simplices of $T$, so for each facet $\tau\in B$ there is a unique vertex $v\in V(T)$ such that $v\in \tau$.   
		
		For each vertex $v\in V(T)$, we define $F(v)$ as follows 
		$$F(v)=\bigcup\{\|\tau\|: \tau\in\text{Facets}(B)\text{ and } v\in\tau\}$$
		
		Note that the geometric representations of the facets of $B$ give a closed cover of $\|K\|$, and so the sets $F(v)$ are also a closed cover. Also note that $gF(v)=F(gv)$ for all $g\in G, v\in V(T)$. 
		
		Define the closed sets $C_1, \dots, C_{\chi(H)}$ as
		$$C_i=\bigcup_{\substack{v\in V(T) \\ c(v)=i}} F(v).$$
		Then the sets $C_i$ are a closed cover for $\|K\|$, and we claim they are a $G$-cover. To see this, suppose by way of contradiction that there is a point $x$ in $\|K\|$, an index $i$, and an element $g\in G$, $g\neq\1$ such that $x,gx\in C_i$. This means there are vertices $u,v\in V(T)$ with $c(u)=c(v)=i$, such that $x\in F(u)$ and $gx\in F(v)$. Thus $gx\in gF(u)=F(gu)$ and $gx\in F(v)$. 
		
		Then, there are simplices $\tau_1, \tau_2\in \text{Facets}(B)$, such that $gu\in\tau_1, v\in \tau_2$ and $gx\in \|\tau_1\|\cap\|\tau_2\|$. Thus, there is a nontrivial face $\tau=\tau_1\cap\tau_2$ in $B$. We claim this means $\{v,gu\}\in T$, and so $v\sim u$ in $H$, but they both had the same color $i$, giving the desired contradiction. 
		
		To finish the proof we establish now our last claim. Since $gu\in\tau_1$, $v\in \tau_2$, and $\tau=\tau_1\cap\tau_2\neq\emptyset$, we can describe them as chains of simplices in $T$:
		\begin{align*}
			\tau_1&=\{gu\prec a_2\prec\dots\prec a_r\}\\
			\tau_2&=\{v\prec b_2\prec\dots\prec b_s\}\text{ and}\\
			\tau&=\{c_1\prec\dots\prec c_t\}
		\end{align*}
		Where each $c_i=a_{i_j}=b_{i_k}$ for some indices. Thus $c_1=a_j=b_k$ for some indices, so $gu\prec c_1$ and $v\prec c_1$, so $\{gu,v\}\subset c_1$. But $c_1$ is a simplex of $T$, so $\{gu,v\}\in T$, as we claimed. 
	\end{proof}
	
	\section[Random G-Borsuk Graphs]{Random $\bm{G}$-Borsuk Graphs}\label{section:random_g_graphs}
	In this section, we focus on the chromatic number of random $G$-Borsuk graphs, and find thresholds for $\epsilon$ depending on $n$, the number of vertices, such that asymptotically almost surely we can more precisely bound its chromatic number. 
	\begin{defi}[Random $G$-Borsuk graphs]
		Let $G$ be a finite group, and $K$ a compact $G$-space. We define the \textbf{random $G$-Graph on $K$}, denoted by $\Ggraph{G}{n,\epsilon}$, as the $G$-Borsuk graph $\Ggraph{G}{X(n),\epsilon}$, where $X(n)=\{x_1,\dots, x_n\}$ is a set of $n$ i.i.d. uniform random points on $K$.
	\end{defi}
	
	Note that the definition of a random $G$-Borsuk graph also depends on the specific underlying $G$-space, however for simplicity we don't include $K$ in the notation and instead make it clear from the context. 
	
	Since a random $G$-Borsuk graph is an induced subgraph of a $G$-Borsuk graph, Theorem \ref{thm:properties_cov} states its chromatic number must be at most $\cov{G}{K}\leq\cov{G}{k}$ (here $k=\ind{G}{K}$), with equality if the random sampling is \textit{dense enough}. In this section we will prove Theorem \ref{thm:chromatic_random}, which establishes, up to a constant,  what $\epsilon(n)$ provides a dense enough sample. On the other hand, when $\epsilon(n)$ is small enough to provide a \textit{sufficiently sparse} sample, then the chromatic number is a.a.s.\ bounded above by $\cov{G}{k-1}$, which if Conjecture \ref{conj:cov_d} is true must be strictly less than $\cov{G}{K}$.

	\begin{theorem}\label{thm:chromatic_random}
		Consider a finite group $G$ and a finite geometric $G$-simplicial complex $K$ of dimension $d$. Suppose $K$ is pure $d$-dimensional, i.e. all of its maximal faces have dimension $d$. And let $k=\ind{G}{K}$. Then, there exist constants $C_1$ and $C_2$ (independent from $n$), such that:
		\begin{enumerate}
			\item If $\epsilon\geq C_1\left(\frac{\log n}{n}\right)^{1/d}$, then a.a.s. 
			$\chi(\Ggraph{G}{n,\epsilon})=\cov{G}{K}.$
			
			\item If $\epsilon\leq C_2\left(\frac{\log n}{n}\right)^{1/k}$, then a.a.s.
			$\chi(\Ggraph{G}{n,\epsilon})\leq\cov{G}{k-1}.$
		\end{enumerate}
	\end{theorem}
	
	Theorem \ref{thm:chromatic_random_intro} follows by taking $K$ a $E_dG$ space, since then $k=\ind{G}{K}=\dim(K)=d$. 
	
	The case $G=\Z_2$ corresponds to Random Borsuk Graphs, which was already established in Theorems 1.1 and 1.2 of \cite{Kahle-Martinez2020}. The main ideas behind the proofs of this generalization remain the same. We first establish some intermediate results before jumping into the proof. 
	
	\subsection{Some Geometric Lemmas}
	
	For the \textit{``sparse case''}, our approach relies on being able to project the vertices of the $G$-Borsuk graph onto the $(d-1)$-skeleton of the classifying space via a $G$-map. This will be possible as long as vertices that where close to each other remain close after the projection. This will follow from the following lemmas studying such projections on planes and $d$-simplices. 
	
	\begin{lemma}\label{lemma:rays_plane} Let $O$ be a point on the plane and let $\ell$ be a line such that the distance from $O$ to $\ell$ is at least $d_1>0$. Let $A$ and $B$ be two points such that the following hold
		\begin{itemize}
			\item $A,B$ and $O$ lie all on the same semi-plane with respect to $\ell$;
			\item the rays $\vec{OA}$ and $\vec{OB}$ intersect $\ell$;
			\item $\dist{O,A'}, \dist{O,B'}\leq d_2$, where  $A'=OA\cap\ell$ and $B'=OB\cap\ell$;
			\item $\dist{O,A}, \dist{O,B}\geq \kappa\epsilon$; and
			\item $\dist{A,B}\leq\epsilon$. 
		\end{itemize}
		Then $$\dist{A',B'}\leq\frac{d_2^2}{\kappa d_1}.$$
	\end{lemma}
	
	\begin{proof}
		Note $$\sin (\measuredangle OA'B')=\frac{\dist{O,\ell}}{\dist{O,A'}}\geq\frac{d_1}{d_2}.$$
		By Law of Sines on the triangle $\Delta AOB$, 
		$$\sin(\measuredangle AOB)=\frac{\dist{A,B}\sin(\measuredangle OBA)}{\dist{O,A}}\leq\frac{\dist{A,B}}{\dist{O,A}}\leq\frac{\epsilon}{\kappa\epsilon}=\frac{1}{\kappa}.$$
		Note $\measuredangle A'OB'=\measuredangle AOB$, so by Law of Sines on the triangle $\Delta A'OB'$ we get
		$$\dist{A',B'}=\frac{\dist{O,B'}\sin(\measuredangle A'OB')}{\sin(\measuredangle OA'B')}\leq\frac{d_2 \frac{1}{\kappa}}{\frac{d_1}{d_2}}=\frac{d_2^2}{\kappa d_1}$$
	\end{proof}
	
	\begin{lemma}\label{lemma:phi_is_lips}
		Let $\Delta$ be a geometric $d$-simplex of diameter at most $d_2$. Let $O$ be an interior point such that $\dist{\partial\Delta,O}\geq d_1$. Take any two points $A, B\in\Delta$ with $\dist{A,B}\leq\epsilon$ and $\dist{O,A}, \dist{O,B}\geq \kappa\epsilon$. Let $A'$ and $B'$ be the intersection of the rays $OA$ and $OB$ with the boundary of $\Delta$ respectively. Then $\dist[\partial\Delta]{A',B'}\leq 2(d+1)d_2^2/(\kappa d_1)$. 
	\end{lemma}
	
	\begin{proof}
		The case $d=1$ is trivial, so we suppose $d\geq2$. Let $\tau_1$ and $\tau_2$ be the $(d-1)$-faces of $\Delta$ that intersect the rays $\vec{OA}$ and $\vec{OB}$, respectively. So $A'\in\tau_1$ and $B'\in\tau_2$. We distinguish two cases:
		\begin{enumerate}
			\item \underline{If $\tau_1=\tau_2$}. Denote $\tau=\tau_1=\tau_2$, so both $A'$ and $B'$ lie on the same $(d-1)$-face $\tau$. In this case, let $\Pi$ be the plane through $O,A$ and $B$. The intersection $\tau\cap\Pi$ must be a line segment containing $A'$ and $B'$. Note that if $\ell$ is the line containing $\tau\cap\Pi$, then $O,A,B$ and $\ell$ satisfy all the hypothesis of Lemma \ref{lemma:rays_plane}, so $$\dist[\partial\Delta]{A',B'}\leq\frac{d_2^2}{\kappa d_1}\leq 2(d+1)\frac{d_2^2}{\kappa d_1}.$$
			\item \underline{If $\tau_1\neq\tau_2$}. Consider the line segment $\overline{AB}$, which lies in the interior of $\Delta$, and let $p:\overline{AB}\to\partial\Delta$ be the projection on $\partial\Delta$ from $O$, i.e. the intersection of rays $\vec{OX}$ for all points $X\in\overline{AB}$. Note $p(\bar{AB})$ will be a polygonal line from $A'\in \tau_1$ to $B'\in \tau_2$, passing through faces $\tau_1=\sigma_1, \sigma_2, \dots, \sigma_r=\tau_2$ of $\partial\Delta$, where $r\leq(d+1)$. Let $E'_i=p(\bar{AB})\cap\sigma_i\cap\sigma_{i+1}$ for each $i=1, \dots, r-1$, and say $E_0'=A'$ and $E_r'=B'$. Correspondingly, let $E_i\in\bar{AB}$ such that $p(E_i)=E_i'$. For each $i=1,\dots, r$, let $\Pi_i$ be the plane through $E_{i-1}, E_{i}$ and $O$. As in the first case, let $\ell_i$ be the line containing $\Pi_i\cap\sigma_i$. Clearly $\dist{E_{i-1},E_i}\leq\epsilon$, and, given $\kappa>1$, we also know $\dist{\bar{AB},O}\geq \frac{\kappa\epsilon}{2}$, so $\dist{E_{i-1},O}, \dist{E_i,O}\geq \kappa\epsilon/2$, then these points and $\ell_i$ satisfy the hypothesis of Lemma \ref{lemma:rays_plane} with $\kappa/2$, so $\dist[\partial\Delta]{E_{i-1}',E_i'}\leq\frac{2d_2^2}{\kappa d_1}$.  Then%
			$$\dist[\partial\Delta]{A',B'}\leq \dist[\partial\Delta]{E_0',E_1'}+\dots+\dist[\partial\Delta]{E_{r-1}',E_r'}\leq r\frac{2d_2^2}{\kappa d_1}\leq 2\frac{(d+1)d_2^2}{\kappa d_1}.$$\qedhere
		\end{enumerate}
	\end{proof}

	We will use the next lemma to translate the volume of balls on $E_kG$ back to $K$, where the random points are drawn. 
	
	\begin{lemma}\label{lemma:volume_k_projections}
		Let $\sigma$ be a geometric $d$-dimensional simplex, $\tau$ a $k$-face of $\sigma$, and $\Psi:\sigma\to\tau$ a simplicial map that leaves $\tau$ invariant. Then there exists a constant $C$ such that for each measurable subset $E\subset\sigma$, we have
		$$\vol[d]{E}\leq C\vol[k]{\Psi(E)}.$$ 
	\end{lemma}
	
	\begin{proof}
		Let $e_1,e_2, \dots,e_d$ be the standard vector basis of $\R^d$, and let $e_0=\vec{0}$ be the origin. For each $m\leq d$, the points $\{e_0,\dots,e_m\}$ are affinely independent, so their convex hull is a geometric $m$-simplex which we denote $\Delta_m$. For each $r\leq m$, let $\pi_r:\Delta_m\to\Delta_r$ be the projection map to the first $r$ coordinates. Note $\pi_r$ can be seen as the simplicial map given by the vertex mapping
		$$\pi_r(e_i)=\begin{cases}
			e_i & \text{ if } i\leq r\\
			e_0 & \text{ if } i>r
		\end{cases}.$$
		
		We start by establishing the result when $k=d-1$. Enumerate the vertices of $\sigma$, $V(\sigma)=\{a_0,\dots, a_d\}$, such that $V(\tau)=\{a_0,\dots, a_{d-1}\}$. Since the map $\Psi$ is determined by its mapping of the vertices, $\Psi(a_i)=a_i$ for all $0\leq 1\leq d-1$ and $\Psi(a_d)$ coincides with one of the other vertices, so shuffling the labels if necessary, $\Psi(a_d)=a_0$. 
		
		Then, there exists a unique affine transformation $\Phi:\sigma\to\Delta_d$ such that $\Phi(a_i)=e_i$ for $0\leq i\leq d$. Let $\phi=\Phi|_\tau$, note then $\phi:\tau\to\Delta_{d-1}$ is an affine map. Moreover, both $\Phi$ and $\phi$ are bijective, so there exist invertible matrices $A, A'$ and vectors $b,b'$ such that $\Phi(x)=Ax+b$ and $\phi(x)=A'x+b'$.  We then get that the following diagram of simplicial maps commutes
		
		\begin{center}
			\begin{tikzcd}[row sep=large,column sep=large]
				\sigma \arrow[]{d}{\Phi} \arrow[]{r}[name=U]{\Psi} 
				& \tau \arrow{d}[swap,shift right=2]{\phi} \arrow[<-,shift left=2]{d}{\phi^{-1}}\\
				\Delta_d \arrow[]{r}[name=D]{\pi} & \phantom{x}\Delta_{d-1}
				\arrow[to path={(D) node[midway,scale=1.5] {$\circlearrowleft$} (U)}]{}
			\end{tikzcd}
		\end{center}
		
		where $\pi=\pi_{d-1}$. In particular, we have $\phi^{-1}\circ\pi\circ\Phi=\Psi$. Let us also point out that for any set $F\subset\Delta_d$, we have $F\subset \pi(F)\times[0,1]$. 
		
		Then, the volume of any measurable set $E\in \sigma$ satisfies:
		\begin{align*}
			\vol[d]{E}&=\frac{\vol[d]{\Phi E}}{|\det A|}\\
			&\leq
			\frac{\vol[d]{(\pi\Phi E)\times[0,1]}}{|\det A|}
			=\frac{\vol[d-1]{\pi\Phi E}\cdot 1}{|\det A|}\\
			&=\frac{|\det A'|}{|\det A|}\vol[d-1]{\phi^{-1}\pi\Phi E}
			=C\vol[d-1]{\Psi E}. 
		\end{align*} 
		
		We now extend the result for any $k<d$. Suppose $V(\sigma)=\{a_0,\dots, a_d\}$ and $V(\tau)=\{a_0,\dots, a_k\}$. For each $k+1\leq i\leq d$, define the $i$-face $\tau_i\subset\sigma$ by $\tau_i=\{a_0,\dots, a_k, a_{k+1},\dots, a_i\}$, so $\tau_k=\tau$ and $\tau_d=\sigma$. Also for each $k+1\leq i\leq d$, define the map $\Psi_i:\tau_i\to\tau_{i-1}$ as the simplicial map given by
		$$\Psi_i(a_j)=\begin{cases}
			a_j, \text{ if } 0\leq j\leq i-1\\
			\Psi(a_i), \text{ if } j=i
		\end{cases}. $$
		
		Thus $\Psi$ factors through the faces $\tau_i$ as depicted in the following commutative diagram
		
		\begin{center}
			\begin{tikzcd}
				\sigma=\tau_d \rar{\Psi_d}
				\arrow[rrr, bend right, "\Psi"]
				& \tau_{d-1} \arrow[r,"\Psi_{d-1}"]  & \cdots\rar{\Psi_{k+1}} & \tau_k=\tau
			\end{tikzcd}
		\end{center}		
		Finally, each $\Psi_i:\tau_i\to\tau_{i-1}$ satisfies the hypothesis of the theorem for the special case that we already proved, so for each $E\subset\sigma=\tau_d$ measurable set, we get
		
		\begin{align*}
			\vol[d]{E}&\leq C_d\vol[d-1]{\Psi_d(E)}\\
			& \leq C_dC_{d-1}\vol[d-2]{\Psi_{d-1}\Psi_d(E)}\\
			&\vdots\\
			&\leq \left(\prod_{i=k+1}^{d}C_i\right) \vol[k]{\Psi_d\Psi_{d-1}\dots\Psi_{k+1}(E)}=C\vol[k]{\Psi(E)}
		\end{align*}
	\end{proof}
	
	\subsection{Other Preliminary Results}
	
	We state the following Lemma without proof. While the proof should be straightforward, it is quite technical so we omit it. Refer to \cite{Bridson1999} for more on geodesics of geometric simplicial complexes, with the intrinsic metric.  
	
	\begin{lemma}\label{lemma:length_chains}
		Let $K$ be a geometric simplicial complex and $\epsilon_0$ small. Then, there exists a constant $N$, depending only on $K$, such that for any $x,y\in K$ with $\dist{x,y}\leq\epsilon_0$, there is a polygonal path from $x$ to $y$ of length $\dist{x,y}$, made up of at most $N$ line segments, each of which is contained in a face of $K$. 
	\end{lemma}
	
	\begin{theorem}\label{thm:homom_boundary}
		Let $G$ be a finite group, and $K$ a finite geometric $G$-simplicial complex of dimension $d$, such that $G$ acts on $K$ via isometries. Moreover, suppose $K$ is pure $d$-dimensional, and let $L=K^{(d-1)}$ its $(d-1)$-skeleton. Fix a constant $d_1>0$, and let $X_0\subset K$ be a collection of one point in the interior of each $d$-face of $K$, such that $gX_0=X_0$ for all $g\in G$, and $\dist{X_0,L}\geq d_1$. Given $\epsilon_0>0$, there exists a constant $\kappa$ depending only on $K$, $\epsilon_0$ and $d_1$, such that for any $\epsilon$ small enough, there is a graph homomorphism 
		
		$$\Phi:\Ggraph{G}{K\setminus B(X_0,\kappa\epsilon),\epsilon}\to \Ggraph{G}{L,\epsilon_0}.$$
		
		Here $B(X_0,\kappa\epsilon):=\bigcup_{x\in X_0} B(x,\kappa\epsilon)$. 
		
	\end{theorem}

	\begin{proof}
		For each $d$-face $\sigma_i\in K$, let $x_i\in X_0$ be the point in $\sigma_i$. For $x\in\sigma_i\setminus\{x_i\}$, define $\phi_i(x)$ to be the unique point of intersection of the ray $\vec{x_ix}$ with the boundary $\partial\sigma_i$. Note this map $\phi_i:\sigma_i\setminus\{x_i\}\to\partial\sigma_i$ is continuous and it restricts to the identity on $\partial\sigma_i$. Moreover, for any $g\in G$, since $g$ acts as a linear map, the ray $\vec{x_ix}$ is mapped to the ray $\vec{g(x_i)g(x)}$, so if $g(x_i)=x_j\in X_0$, we get  
		$$g\left(\phi_{i}(x)\right)=\phi_{j}\left(g(x)\right).$$
		
		Let $K_d$ be the collection of $d$-faces of $K$, and let $d_2=\max_{\sigma\in K_d} \diam{\sigma}$.  Let $N$ be the constant from Lemma \ref{lemma:length_chains} corresponding to $\epsilon_0$. Define $\kappa=2\frac{N(d+1)d_2^2}{d_1\epsilon_0}$. Let $\epsilon<\epsilon_0$ be small enough so that $B(x_i,\kappa\epsilon)\subset\text{interior}(\sigma_i)$ for all $i$. Then, by Lemma \ref{lemma:phi_is_lips}, for each $i$ and $x, y\in \sigma_i\setminus B(x_i,\kappa\epsilon)$ such that $\dist[K]{x,y}\leq\epsilon$ we get $\dist[L]{\phi_i(x),\phi_i(y)}\leq\epsilon_0/N$.
		
		Define $\Phi:K\setminus B(X_0,\kappa\epsilon)\to L$ by pasting all the maps $\phi_i$. That is, if $x\in\interior{\sigma_i}$ for some $i$, let $\Phi(x)=\phi_i(x)$, and if $x\in L$, let $\Phi(x)=x$. Since all the maps $\phi_i$ are continuous and they all restrict to the identity on $L$, pasting them produces a continuous map. Moreover, $\Phi$ is a $G$-map:
		$$\Phi(g(x))=\phi_i(g(x))=g(\phi_j(x))=g\Phi(x),$$
		where $i$ and $j$ are the indices such that $g(x)\in\sigma_i$ and $x\in\sigma_j$; and $\Phi(g(x))$ is defined since $x\in \sigma_i\setminus B(x_i,\kappa\epsilon)$ and $g$ is an isometry, so $g(x)\in\sigma_j\setminus B(x_j,\kappa\epsilon)$.
			
		We now consider the induced map $$\Phi:V(\Ggraph{G}{K\setminus B(X_0,\kappa\epsilon),\epsilon})\to V(\Ggraph{G}{L,\epsilon_0}),$$ by checking that it respects the adjacency of vertices, we will have a graph homomorphism as desired. Note that to do this is enough to prove that $\dist[L]{\Phi(x),\Phi(y)}\leq\epsilon_0$ whenever $\dist[K]{x,y}\leq \epsilon$, and $x,y\in K\setminus B(X_0,\kappa\epsilon)$. Indeed, if $x\sim y$ in $\Ggraph{G}{K\setminus B(x_0,\kappa\epsilon),\epsilon}$ then there is a $g\in G$, $g\neq \1$, such that $\dist[K]{x,gy}, x\in K\setminus B(X_0,\kappa\epsilon)$ and since $g$ is an isometry, $gy\in K\setminus B(X_0,\kappa\epsilon)$. Thus if our claim holds, $\dist[L]{\Phi(x),g\Phi(y)}=\dist[L]{\Phi(x),\Phi(gy)}\leq\epsilon_0$ and $\Phi(x)\sim\Phi(y)$ in $\Ggraph{G}{L,\epsilon_0}$ as desired.\\
		
		We now prove the claim. Let $x,y\in K\setminus B(X_0,\kappa\epsilon)$, such that $\dist[K]{x,y}\leq\epsilon<\epsilon_0$. Then, there is a shortest polygonal path in $K$ between $x$ and $y$ of length less than $\epsilon$. This path passes through the $d$-faces $\tau_1, \tau_2, \dots, \tau_r$, where $x\in\text{interior}(\tau_1)$, $y\in\text{interior}(\tau_r)$ and $r\leq N$ by Lemma \ref{lemma:length_chains}. For each $i=1, \dots, r-1$, let $z_{i}$ be the intersection of this shortest path with $\tau_i\cap\tau_{i+1}$, let $z_0=x$ and $z_{r}=y$. Note $\dist[K]{z_i,z_{i+1}}\leq\epsilon$ for all $i=0, \dots r-1$, and $x_i\in L$ for $i\neq 0, r+1$, so $\dist[K]{X_0,z_i}>\kappa\epsilon$. 
		
		Take $i\in\{1, \dots, r\}$. If $\tau_i=\sigma_j$ is a $d$-face of $K$, as we noted before $z_{i-1}, z_{i}\in\sigma_j\setminus B(x_j,\kappa\epsilon)$ and\ $\dist[K]{z_{i-1},z_i}\leq\epsilon$, so $\dist[L]{\Phi(z_{i-1}), \Phi(z_i)}=\dist[L]{\phi(z_{i-1}),\phi(z_i)}\leq \frac{2(d+1)d_2^2}{\kappa d_1}\leq \frac{\epsilon_0}{N}$. If instead $\tau_i\in L$, there is some $d$-face $\sigma_j$ such that $\tau_i\leq\sigma_j$, and again $z_{i-1}, z_i\in\sigma_j\setminus B(x_j,\kappa\epsilon)$ and$\dist[K]{z_{i-1},z_i}= \epsilon$, so $\dist[L]{\Phi(z_{i-1}),\Phi(z_i)}\leq \frac{\epsilon_0}{N}$. Therefore, %
		$$\dist[L]{\Phi(x),\Phi(y)}\leq\sum_{i=1}^{r}\dist[L]{\Phi(z_{i-1}),\Phi(z_{i})}\leq \epsilon_0,$$
		which finishes the proof. 		
	\end{proof}
	
	Recall that a \textbf{$\bm{\delta}$-net} $\calB$ of a compact metric space $K$, is a finite subset such that for each $x\in K$, there is a $y\in\calB$ such that $\dist[K]{x,y}<\delta$. Similarly, a \textbf{$\bm{\delta}$-apart set} $\calB$ of $K$, is a finite subset such that for any $x,y\in\calB$ with $x\neq y$, $\dist[K]{x,y}>\delta$. The next lemma bounds the cardinality of $\delta$-nets and $\delta$-apart sets for geometric simplices, the technique mimics what we did for $d$-Spheres on \cite{Kahle-Martinez2020}.  
	
	\begin{lemma}\label{lemma:delta-nets}
		Let $\Delta$ be a geometric $d$-simplex. Then there exist constants $C_1(\Delta)$ and $C_2(\Delta)$ such that for every $\delta>0$ small enough, there is a $\delta$-net $\calB$ that is also a $\delta$-apart set, such that 
		$$\frac{C_1(\Delta)}{\delta^d}\leq |\calB|\leq \frac{C_2(\Delta)}{\delta^d}$$
	\end{lemma}
	
	\begin{proof}
		Regard $\Delta$ as embedded in $\R^d$, with its incenter at the origin. Let $\delta>0$ small. Since $\Delta$ is compact, we can construct $\calB$ inductively. Indeed, choose any
		point $y_1 \in\Delta$. Then, for each $n\geq2$, if $B_n = \cup_{i=1}^n B(y_i, \delta) \subsetneq \Delta$, choose any $y_{n+1} \in \Delta B_m$. Otherwise, stop
		and let $\calB = \{y_1, y_2,\dots, y_n\}$. The process stops by compactness. Clearly $\calB$ is a $\delta$-net and a $\delta$-apart set. Note $\Delta\subset\bigcup B(y_i,\delta)$, thus
		\begin{align*}
			\vol{\Delta} \leq\vol{\bigcup B(y_i,\delta)}&\leq |\calB|\vol{B(y_i,\delta)}=|\calB|\omega_d\delta^d\text{, so}\\
			\frac{\vol{\Delta}}{\omega_d\delta^d} &\leq |\calB|.
		\end{align*}  
		Here $\omega_d$ is the volume of the unit $d$-ball. 
		If $\delta$ is smaller than the inradius of $\Delta$, we always have $B(x,\delta/2)\subset 2\Delta$, for all $x\in \Delta$. Since the $y_i's$ are a $\delta$-apart set, clearly $B(y_i,\delta/2)\cap B(y_j,\delta/2)=\emptyset$ whenever $i\neq j$, and so we have	
		\begin{align*}
			2^d\vol{\Delta}=\vol{2\Delta} \geq \vol{\biguplus B(y_i,\delta/2)}&=|\calB|\vol{B(y_i,\delta/2)}=|\calB|\frac{\omega_d\delta^d}{2^d}\text{, so}\\
			\frac{4^d\vol{\Delta}}{\omega_d\delta^d} &\geq|\calB|.	
		\end{align*}
		Thus setting $C_1(\Delta)=\frac{\vol{\Delta}}{\omega_d}$ and $C_2(\Delta)=\frac{4^d\vol{\Delta}}{\omega_d}$, we get the result. 	
	\end{proof}
	
	\begin{lemma}\label{lemma:g_map_lipschitz_volumes}
		Let $G$ be a finite group and $K$ a finite geometric $G$-simplicial complex. Denote $k=\ind{G}{K}\leq\text{dim}(K)=d$. Then, there exists a Lipschitz $G$-simplicial map $\Phi:K\xrightarrow{G}E_k G$ and a constant $\bar{C}$ such that
		$$\vol[d]{\Phi^{-1}\left(B(x,\delta)\right)}\leq\bar{C}\vol[k]{B(x,\delta)}$$ for all $x$ and $\delta$ such that $B(x,\delta)$ is contained in the interior of a $k$-face of $E_k G$. 
	\end{lemma}
	
	\begin{proof}
		Since $k=\ind{G}{K}$, by definition there is some $G$-map $\Phi:K\xrightarrow{G}\ E_k G$. Moreover, by the simplicial approximation theorem, we may assume that $\Phi$ is a simplicial map, substituting $K$ and $E_kG$ by some finer subdivisions if necessary. While the more general simplicial approximation theorem doesn't guarantee that we still get a $G$-map, this can be achieved by the standard procedure of defining it on a representative of each $G$-orbit and then extending it appropriately (see e.g. \cite{Matousek2008,Kozlov2008}). Since a simplicial map induces a map between the geometric complexes via linear extensions,  $\Phi:K\xrightarrow{G} E_kG$ is just a piece-wise linear map, determined by the images of the vertices of $K$. As such, and since linear maps are continuous iff they are bounded, $\Phi$ must be a Lipschitz map. 
		
		Suppose now that $B(x,\delta)$ is contained in the interior of $k$-face $\tau$ of $E_kG$. Since $\Phi$ is a simplicial map, $\Phi^{-1}(\tau)$ is a sub-simplicial complex of $K$. Clearly the $d$-volume of $\Phi^{-1}(B(x,\delta))\cap\sigma$ is zero whenever $\dim(\sigma)<d$. At the same time, if $\sigma$ is a face of $\Phi^{-1}(\tau)$ such that $\dim(\Phi(\sigma))<k$, then $B(x,\delta)\cap\Phi(\sigma)=\emptyset$ so $\Phi^{-1}(B(x,\delta))\cap\sigma=\emptyset$. Hence, the $d$-volume of $\Phi^{-1}(B(x,\delta))$ only depends on it intersection with the $d$-faces of $\Phi^{-1}(\tau)$ that are mapped onto $\tau$ surjectively. 
		
		Denote $K_d(\tau)=\left\{\sigma\in K: \dim(\sigma)=d\text{ and }\Phi(\sigma)=\tau\right\}$ and let $\sigma\in K_d(\tau)$. Enumerate $\sigma=\{a_0, a_1,\dots, a_d\}$ such that if $\omega=\{a_0,\dots, a_k\}$, then $\Phi(\omega)=\tau$. Thus, for each $a_i$ with $i>k$, $\Phi(a_i)=\Phi(a_j)$ for exactly one index $0\leq j\leq k$. Let $\Psi:\sigma\to\omega$ be the simplicial map given by the assignment of vertices 
		$$\Psi(a_i)=\begin{cases}
			a_i, \text{ if } 0\leq i\leq k\\
			a_j, \text{ s.t. } \Phi(a_i)=\Phi(a_j), \text{ if } k+1\leq i\leq d
		\end{cases},$$
		and let $\varphi=\Phi|_\omega$ the restriction of $\Phi$ to $\omega$. Note that these two maps give the factorization $\Phi=\Psi\circ\varphi$. Note also that $\varphi$ is a bijective affine transformation, so there exists an invertible matrix $A$ and a vector $b$ such that $\varphi(x)=Ax+b$. Let $E(\sigma)=\Phi^{-1}(B(x,\delta))\cap\sigma$, clearly $E(\sigma)$ is a measurable subset, and applying Lemma \ref{lemma:volume_k_projections} to $\sigma$, $\omega$ and the map $\Psi$, there exists a constant $C(\sigma)$ such that
		
		\begin{align*}
			\vol[d]{E} 
			&\leq C(\sigma)\vol[k]{\Psi(E)}
			=\frac{C(\sigma)}{|\det A|}\vol[k]{\varphi\Psi(E)}\\
			&=\frac{C(\sigma)}{|\det A|}\vol[k]{\Phi(E)}
			\leq \frac{C(\sigma)}{|\det A|}\vol[k]{B(x,\delta)}\\
			&=\bar{C}(\sigma)\vol[k]{B(x,\delta)}.
		\end{align*}
		
		Thus, letting $$\hat{C}(\tau)=\sum_{\sigma\in K_d(\tau)}\bar{C}(\sigma),\phantom{xx}\text{and} \phantom{xx} \bar{C}=\max_{\substack{\tau\in E_kG\\ \dim(\tau)=k}}\hat{C}(\tau),$$ we get the desired inequality. 
	\end{proof}
	
	\subsection{Proof of Theorem \ref{thm:chromatic_random}}
	
	\subsubsection{\textit{Dense} Case}
	
	If $\epsilon\to 0$ slowly enough, with high probability the random vertices will form a small net, and so Theorem \ref{thm:rel_graph_cov} will guarantee that the chromatic number is the same as $\cov{G}{K}$. The argument behind is a union bound, and the independence of the random vertices. 
	
	\begin{proof}[Proof of dense case]
		Note that as soon as $\epsilon(n)$ is small enough, Theorem \ref{thm:rel_graph_cov} item 2 gives the upper bound $$\chi(\Ggraph{G}{n,\epsilon})\leq\cov{G}{K}.$$
		
		Let $K_d=\{\sigma\in K | \dim(\sigma)=d\}$ be the set of $d$-faces of $K$. Thus, since $K$ is pure $d$-dimensional, we have $\lVert K\rVert=\left\lVert\bigcup_{\sigma\in K_d}\sigma\right\rVert$. By Lemma \ref{lemma:delta-nets}, for each $d$-face $\sigma\in K_d$, there exist constants $\overline{C}(\sigma)$ such that for any given $\delta>0$ small, we can find a $\delta$-net $\calB(\sigma)$ of $\sigma$ such that $|\calB(\sigma)|\leq \overline{C}(\sigma)\delta^{-d}$. Define a global $\overline{C}:=\sum_{\sigma\in K_d} \overline{C}(\sigma)$. Hence, by defining $\calB=\bigcup_{\sigma\in K_d} \calB(\sigma)$ we can always find a $\delta$-net  of $K$ such that $|\calB|\leq \bar{C}\delta^{-d}$. 
		
		Similarly, the volume of a $d$-ball of radius $\delta$ is proportional to $\delta^d$. And it is known that for a $d$-face $\sigma$, if $\delta$ is small enough and $x\in \sigma$ then  
		$$V(B(x,\delta)\cap \sigma)\geq c(\sigma)\delta^d,$$ for a constant $c(\sigma)>0$ that depends on the \textit{hyper-angle} of the sharpest corner of $\sigma$. 
		Define $c:=\frac{\min_{\sigma\in K_d}c(\sigma)}{V(K)}$, thus for any point $x\in K$ and $\delta>0$ small, we have $$\frac{V(B(x,\delta)\cap K)}{V(K)}\geq\frac{V(B(x,\delta)\cap\sigma)}{V(K)}\geq\frac{c(\sigma)}{V(K)}\delta^d\geq c\delta^d.$$
		
		Let $D$ be as in Theorem \ref{thm:rel_graph_cov} item 1. Then, we claim that  $C_1=\frac{2D}{\sqrt[d]{c}}$ is as stated in the theorem. That is, if $\epsilon\geq C_1(\log n/n)^{1/d}$, then we will show that a.a.s. $\cov{G}{K}\leq\chi\left(\Ggraph{G}{n,\epsilon}\right)$.
		
		For each $n$, let $\delta=\epsilon(n)/(2D)$, and consider the corresponding $\delta$-net of $K$, $\calB=\{y_1,\dots, y_N\}$. For each $i=1,\dots, N$ define the closed set $F_i=B(y_i,\delta)\cap K$. Then $K=\cup_{i=1}^N F_i$, $\frac{V(F_i)}{V(K)}\geq c\delta^d$, and $N\leq\bar{C}\delta^{-d}\leq \bar{C}c(n/\log n)$  as discussed above.  
		
		Let $X=\{x_1,\cdots, x_n\}$ be the vertices of the random $G$-Borsuk graph, i.e. they are i.i.d. uniform random points on $K$. The computations below will establish that a.a.s. $X$ is a $(\epsilon/D)-net$ of $K$, and thus Theorem \ref{thm:rel_graph_cov} item 1 gives the desired lower bound for the chromatic number. 
		
		For any $i=1,\dots, N$ and $j=1,\dots, n$ we have	
		\begin{align*}
			\P{x_j\in F_i} & =\frac{V(F_i)}{V(K)}\geq c\delta^d=c\frac{\epsilon^d}{(2D)^d} \geq \frac{\log n}{n}.
		\end{align*}
		
		Then, the probability that some of the sets $F_i$ doesn't contain any vertex of $X$ is
		\begin{align*}
			\P{\bigvee_{i=1}^N(X\cap F_i=\emptyset)} &\leq \sum_{i=1}^N\P{X\cap F_i=\emptyset} = \sum_{i=1}^N\P{x_j\not\in F_i}^n\\
			& \leq \sum_{i=1}^N(1-\log n/n)^n\leq N(1-\log n/n)^n\\
			&\leq \bar{C}c\frac{n}{\log n}e^{-\log n}=\frac{\bar{C}c}{\log n}\to 0 \text{ as }n\to\infty.
		\end{align*}
		Thus, a.a.s. every set $F_i$ will contain some vertex of $X$, and hence the set $X$ is a $2\delta=(\epsilon/D)$-net of $K$ as desired.  
	\end{proof}
	
	\subsubsection{\textit{Sparse} Case}
	When $\epsilon\to 0$ fast, and if $K=E_kG$, a.a.s. there will be a point $x_0\in K$, not too close to $K^{(k-1)}$, such that its orbit under $G$ is always at a distance at least $\kappa\epsilon$ of all the random vertices of $\Ggraph{G}{n,\epsilon}$. In this case, Theorem \ref{thm:homom_boundary} provides a graph homomorphism from the random $G$-Borsuk graph on $K$ to a $G$-Borsuk graph on $K^{(k-1)}$, and so Theorem \ref{thm:rel_graph_cov} gives the desired upper bound. When $K$ is not a finite classifying space, the same will still be true by looking at the pre-images of the $G$-map $\Phi:K\to E_kG$.  
	
	The idea behind the proof is to consider the pre-images of open balls around many orbits $Gx_0$ in $E_kG$, and prove that a.a.s. there will be at least one of them that doesn't contain any vertex of $\Ggraph{G}{n,\epsilon}$. To get spacial independence of random events, we consider a \textit{Poissonization} of the random vertices, i.e. we draw them via a Poisson Point Process instead. We include a couple of results about Poisson Point Processes and the Poisson distribution without proof, for their proofs and a complete discussion refer to \cite{Penrose2003} or \cite{Kingman1993}. 
	
	\begin{theorem}[Poissonization]\label{thm:poissonization}
		Let $x_1, x_2, \dots$, be uniform random variables on a compact metric space $K$. Let $m\sim \text{Pois}(\lambda)$ and let $\eta$ be the random counting measure associated to the point process $P_\lambda$= $\{x_1,x_2,\dots,x_m\}$. Then $P_\lambda$ is a Poisson Point Process and for a Borel \mbox{$A\subset K$, $\eta(A)\sim \pois{\lambda\frac{V(A)}{V(K)}}$}.
	\end{theorem}
	
	\begin{lemma}\label{lemma:poisson}
		For $n\geq 0$, $\P{\pois{2n}<n}\leq e^{-0.306n}.$
	\end{lemma}
	
	To make our proofs cleaner, we address the probabilistic part in the following lemma. 
	
	\begin{lemma}\label{lemma:poisson_empty}
		Let $K$ be a compact metric space pure $d$-dimensional. Let $D$ be a constant such that for each $n$, there are $N=N(n)$ disjoint closed sets $F_1, \dots, F_N\subset K$ such that
		\begin{enumerate}
			\item $\frac{\vol{F_r}}{\vol{K}}\leq \frac{\log n}{3n}$ for all $r=1,\dots, N$; and
			\item $N\geq \frac{D n}{\log n}$.
		\end{enumerate}
		Let $X_n=\{x_1,\dots, x_n\}$ be a collection of $n$ i.i.d. uniform random points on $K$ with corresponding counting measure $\eta_{1}^n$. Then 
		$$\P{\bigvee_{r=1}^N\eta_{1}^n(F_i)=0}\to1\text{ as }n\to\infty.$$
	\end{lemma}

	\begin{proof}
		We use the fact that the $F_r$'s are disjoint, by considering instead the \textit{Poissonization} of the random points, so that we get spatial independence. Hence let $x_1,x_2, \dots$ be uniform random variables on $K$, let $M\sim\pois{2n}$, and consider the Poisson Point Process $\{x_1,\dots, x_M\}$, as described in Theorem \ref{thm:poissonization}. Letting $\eta$ be the corresponding counting measure, we get
		
		\begin{align*}
			\P{\bigwedge_{r=1}^N\eta(F_r)>0} & =\prod_{r=1}^{N}\P{\eta(F_r)>0}\\
			& =\prod_{r=1}^{N}\left(1-\P{\pois{\frac{2n\vol{
							F_r}}{\vol{K}}}=0}\right)\\
			& =\prod_{r=1}^{N}\left(1-e^{-\frac{2n\vol{F_r}}{\vol{K}}}\right) \leq
			\left(1-e^{-(2/3)\log n}\right)^N\\
			& \leq \exp\left(-Nn^{-2/3}\right)\leq \exp\left(-\frac{Dn^{1/3}}{\log n}\right)\to 0\text{ as }n\to\infty.
		\end{align*}
		
		Thus $\P{\bigvee_{r=1}^N\eta(F_r)=0}\to1$ as $n\to\infty$. Clearly when $M\geq n$, $\eta(F_r)=0$ implies $\eta_{1}^n(F_r)=0$ as well. Thus, applying Lemma \ref{lemma:poisson}, we get 
		\begin{align*}
			\P{\bigvee_{r=1}^N\eta(F_r)=0}&\leq \P{\bigvee_{r=1}^N\eta_{1}^n(F_r)=0}+\P{M<n}\\
			& \leq\P{\bigvee_{r=1}^N\eta_{1}^n(F_r)=0}+o(1).
		\end{align*}
		From where the desired result follows. 	
	\end{proof}

	\begin{proof}[Proof of Sparse Case]
		Let $M$ be a geometric realization of $E_kG$ such that $G$ acts on $M$ via isometries. Then, let $\Phi:K\xrightarrow{G}M$ be the simplicial $G$-map and $C$ the corresponding constant given by Lemma \ref{lemma:g_map_lipschitz_volumes}. Let $\lambda$ be the Lipschitz constant of $\Phi$. 
		
		Let $L=M^{(k-1)}$ be the $(k-1)$-skeleton of $M$. Restricting the action of $G$ to $L$ we also get a geometric $G$-simplicial complex. Moreover, $\dim(L)=k-1$ and, since $M$ is $k$-connected, $L$ must be $(k-1)$-connected, thus $L$ is a $(k-1)$-classifying space of $G$, and $\cov{G}{L}=\cov{G}{k-1}$. Let $\epsilon_0>0$ be as in Theorem \ref{thm:rel_graph_cov}, that is, for any $X\subset L$ and any $0<\epsilon\leq\epsilon_0$, 
		$$\chi\left(\Ggraph{G}{X,\epsilon}\right)\leq\cov{G}{L}=\cov{G}{k-1}.$$ 
		
		Let $M_d$ be the set of $d$-faces of $M$. Let $d_1=(1-\rho)\min_{\sigma\in M_d}{\text{inradius}(\sigma)}$,
		for $\rho<1$ any fixed constant close to 1. And let $\kappa$ be the constant given by Theorem \ref{thm:homom_boundary}, corresponding to $\epsilon_0$ and $d_1$. Partition the $d$-faces of $M$ into $G$-orbits, and let $\calC=\{\sigma_1,\dots\sigma_m\}$ be a collection of one representative per orbit, where $m=|\calC|$. 
		
		We will prove that the constant $$C_2=\frac{1}{\kappa\lambda}\left(\frac{\vol[d]{K}}{3m|G|C\omega_k}\right)^{1/k},$$ where $\omega_k$ is the volume of the unit $k$-ball, is as stated in the theorem. That is, if $\epsilon(n)\leq C_2(\log n/n)^{1/k}$, then a.a.s. $\chi(\Ggraph{G}{n,\epsilon})\leq\cov{G}{L}$. 
		
		For a $d$-face $\sigma$, denote by $\rho\sigma$ the homothetic copy of $\sigma$ obtained by scaling it by a factor of $\rho$ with respect to its incenter. Note that $$\dist{\rho\sigma,\partial\sigma}=(1-\rho)\text{inradius}(\sigma)\geq d_1.$$
		
		Fix $n$ large enough so that $\epsilon$ satisfies $\kappa\lambda\epsilon\ll d_1$, and let $\delta=2\kappa\lambda\epsilon$. For each $1\leq i\leq m$, Lemma~\ref{lemma:delta-nets} guarantees there is a constant $\bar{C}(\rho\sigma_i)$ and a $\delta$-apart set $\calB_i=\{y_1^i,\dots, y_{|\calB_i|}^i\}\subset\rho\sigma_i$ such that, $$\frac{\bar{C}(\rho\sigma_i)}{\delta^k}\leq |\calB_i|.$$
		
		Let $\bar{C}=\min_i{\bar{C}(\rho\sigma_i)}$ and $N=\min_i |\calB_i|$, then $\bar{C}/\delta^k\leq N$.
		
		For each $r=1,\dots, N$ construct the set $Y_r$, as $$Y_r:=\{gy_r^i: g\in G, 1\leq i\leq m\}.$$
		
		Note this construction guarantees $Y_r=gY_r$ for all $g\in G$, it has one point in the interior of each $d$-face of $M$, and $\dist{Y_r,L}\geq d_1$. Thus, for each $1\leq r\leq N$, Theorem \ref{thm:homom_boundary} guarantees that there is a graph homomorphism $$\Ggraph{G}{M\setminus B(Y_r,\kappa\lambda\epsilon),\lambda\epsilon}\to\Ggraph{G}{L,\epsilon_0}.$$ 
		Note that because the points in $Y_r$ come from $\delta$-apart sets and $\kappa\lambda\epsilon<d_1$, then $B(Y_r,\kappa\lambda\epsilon)\cap B(Y_{r'},\kappa\lambda\epsilon)=\emptyset$ whenever $r\neq r'$, and $B(Y_r,\kappa\lambda\epsilon)$ is a disjoint union of $k$-balls. 
		
		For each $r=1,\dots N$, let $F_r=\Phi^{-1}\left(B(Y_r,\kappa\lambda\epsilon)\right)$, thus by Lemma \ref{lemma:g_map_lipschitz_volumes}, we get
		\begin{align*}
			\vol[d]{F_r}&=\sum_{y\in Y_r}\vol[d]{\Phi^{-1}\left(B(y,\kappa\lambda\epsilon)\right)}
			\leq \sum_{y\in Y_r}C\vol[k]{B(y,\kappa\lambda\epsilon)}\\
			&= |Y_r|C\omega_k(\kappa\lambda\epsilon)^k
			\leq m|G|C\omega_k(\kappa\lambda C_2)^k\left(\frac{\log n}{n}\right)\\
			&\leq \frac{\vol[d]{K}\log n}{3n}.
		\end{align*}
		
		So $\frac{\vol[d]{F_r}}{\vol[d]{K}}\leq\frac{\log n}{3n}$ for all $r=1,\dots N$, and we also have 
		$$N\geq\frac{\bar{C}}{(2\kappa\lambda\epsilon)^k}\geq\frac{\bar{C}}{(2\kappa\lambda C_2)^k}\cdot\frac{n}{\log n}=C'\frac{n}{\log n}.$$
		
		Thus we can apply Lemma \ref{lemma:poisson_empty} to $X_n\subset K$, $n$ i.i.d. uniform random points on $K$, so a.a.s. there exists some index $r$ such that $X_n\cap F_r=\emptyset$. Finally, since $\Phi$ is a Lipschitz $G$-map, by Theorem~\ref{thm:properties_cov} item 2, it induces a graph homomorphism, and so we have the following composition of graph homomorphisms
		\begin{align*}
			\Ggraph{G}{X_n,\epsilon}\hookrightarrow\Ggraph{G}{K\setminus F_r,\epsilon}\xrightarrow{\Phi}\Ggraph{G}{M\setminus B(Y_r,\kappa\lambda\epsilon),\lambda\epsilon}\to\Ggraph{G}{L,\epsilon_0}
		\end{align*}
		
		So we get the desired result		
		$$\chi\left(\Ggraph{G}{n,\epsilon}\right)\leq\chi\left(\Ggraph{G}{L,\epsilon_0}\right)\leq\cov{G}{L}=\cov{G}{k-1}.$$	
	\end{proof}
	
	\section{Further Questions}
	\begin{itemize}
		\item 	While we have some computer-aided examples suggesting that Conjecture \ref{conj:cov_d} might be true, it is still open. Thus, natural follow up questions would be to prove or disprove Conjectures \ref{conj:smaller_conjectures} and \ref{conj:cov_d}. Towards this end, it would be interesting to produce $G$-covers for more small groups in dimensions 2 and 3. In particular, it would be useful to be able to compute $\cov{\Z_m}{2}$ and $\cov{\Z_m}{3}$ with more efficient algorithms, rather than solving an integer programming problem.
		\item  	For the random $G$-Borsuk graphs, we wonder whether there exist similar thresholds to that in Theorem \ref{thm:chromatic_random_intro} for each possible chromatic number. We state this question as a conjecture. 
		
		\begin{conj}
			Let $G$ be a finite group, and $K$ a geometric $E_dG$ space. Then, for each $1\leq k\leq d+1$, there exist constants $C_k, \bar{C}_k$ and functions $f_k:\N\to\R^+$, such that:
			\begin{center}
				If $C_kf_k(n)\leq\epsilon(n)\leq\bar{C}_kf_{k+1}(n)$, then a.a.s.\ $\chi\left(\Ggraph{G}{n,\epsilon}\right)=\cov{G}{k}$. 
			\end{center}
		\end{conj}
		
		In particular, Theorem \ref{thm:chromatic_random_intro} shows the conjecture is true for $k=d$, taking $f_d(n)=\left(\log n/n\right)^{1/d}$ and $f_{d+1}(n)=\epsilon_0$ from Theorem \ref{thm:rel_graph_cov}. 
	\end{itemize}
	
	\section{Acknowledgments}
	The author wants to thank his advisor Matthew Kahle, for fruitful conversations, frequent words of encouragement, and for kindly proof-reading and suggesting improvements to parts of this paper. 
	
	\appendix
	\section{Computer-aided Examples}
	In this appendix we describe the approach taken to compute the values of Table \ref{table:covering_numbers} and the graphs on Figure \ref{fig:upper_bounds}. Let us start by describing a general method for any finite group $G$, and then we will discuss our specific approach in dimensions 2 and 3 for cyclic groups. 
	
	\subsection*{General approach}
	
	\begin{enumerate}
		\item Suppose $L$ is a $E_dG$ classifying space where $G$ acts via isometries. Moreover, suppose we know $\cov{G}{d}=|G|+d$ and we have a corresponding closed $G$-cover of $L$.  
		\item Take $K=G*L$, so $K$ is a $E_{d+1}G$ classifying space, and $G$ acts on $K$ via isometries.
		\item Let $\Delta$ be a \textit{subdivision operation} that respects the $G$-action, that is, an algorithm that takes $K$ and produces a $G$-simplicial complex $\Delta K$ with the same geometric realization $\|\Delta K\|=\|K\|$, with $V(K)\subset V(\Delta K)$, and such that the action of $G$ restricted to $\Delta K$ is again a $G$-action via simplicial maps. Denote by $\Delta_k K$ the repeated application of $\Delta$. Note that we can realize $L\hookrightarrow K$, so $\Delta$ restricted to $L$ also produces a subdivision of $L$
		\item For $T=\Delta_k K$, let $H(T)$ be the graph with vertices $V(T)$ and edges $u\sim v$ whenever there is a $g\in G$, $g\neq\1$ s.t. $\{u,gv\}\in T$. Similarly, we define $H(T_0)$ for $T_0=\Delta_k L$. 
		\item Let $k$ be the minimum integer such that the known $G$-cover for $L$ produces a proper graph coloring for $H(T_0)$, where $T_0=\Delta_k L$. Thus $\chi(H(T_0))=\cov{G}{d}=|G|+d$. 
		\item Let $c_0:V(T_0)\to\{1,\dots, |G|+d\}$ be this coloring.  
		\item Let $T=\Delta_k K$, and $H=H(T)$. 
		We define a partial coloring $c:V(T)\to\{1,\dots |G|+d, |G|+d+1\}$, as follows 
		$$c(v)=\begin{cases}
			c_0(v_0)&\text{ if } v_0 \in T\cap L\\
			|G|+d+1 &\text{ if } v=t\1\oplus (1-t)v_0\text{ for some }t>0, v_0\in V(T_0)
		\end{cases}$$
		\item We then attempt to extend this partial coloring to the remaining vertices of $H$: $\{tg\oplus (1-t)v_0: g\in G, g\neq\1, v_0\in V(T_0), 0<t\leq 1\}$. We do this by solving the corresponding integer programming problem. If this is possible, we continue with the next step. Otherwise, we increase $k$ by one and repeat steps 6-7. 
		\item If the partial coloring can be extended to a proper coloring for $H$, Theorem \ref{thm:general_upper_bound} guarantees $\cov{G}{d+1}\leq \chi(H)=|G|+d+1$, and as in its proof, coloring the Barycentric Subdivision with the  corresponding colors of the vertices of $T$, provides a $G$-cover for $K$. 
		\item This whole process may be repeated for dimension $d+2$, using the $G$-cover of $K$ found as the starting point. 
	\end{enumerate}
	
	\subsection*{For Cyclic Groups}
	\subsubsection*{2D Case}
	
	Let $\Z_m=\{\1, \nu, \nu^2, \dots, \nu^{m-1}\}$ be a cyclic group of order $m$. Let $C_m$ be the cycle of length $m$, with its vertices labeled by the elements of $\Z_m$. It is convenient to think of the rational points of $C_m$ as labeled by elements in the module $\Q[\Z_m]$. That is, the rational points in the segment $[\nu^i, \nu^{i+1}]$ can be described as $\alpha\nu^i + (1-\alpha)\nu^{i+1}$, for $0\leq\alpha\leq 1$, $\alpha\in \Q$. Thus, taking left multiplication in the module $\Q[\Z_m]$ corresponds to the $\Z_m$-action on $\|C_m\|$. 
	
	\begin{enumerate}
		\item We take $E_1\Z_m=\S^1=C_m$, so $\cov{\Z_m}{1}=m+1$ and can be realized by dividing $\S^1$ into $(m+1)$-arcs of the same size. 
		\item Consider $E_2=\Z_m*C_m$. Its vertices are of the form $\alpha g\oplus (1-\alpha) g$ for $g\in G$ and $\alpha\in\{0,1\}$. Similarly, we can label all rational points in $\|\Z_m*C_m\|$ by elements of the module $\Q[\Z_m]^2$. That is, $x=\alpha g \oplus (1-\alpha) b$ for $0\leq\alpha\leq 1$, $\alpha\in\Q$, $g\in \Z_m$ and $b$ a rational point in $\|C_m\|$. 
		\item For the subdivision operation, we take $\Delta$ to be the \textit{medial triangle subdivision}. This process subdivides each triangle as follows:
		\begin{enumerate}
			\item Take a triangle $\{v_1,v_2,v_3\}$, with labels $v_i=a_i\oplus b_i$ for $a_i, b_i\in \Q[\Z_m]$. 
			\item Let $v_{ij}=\frac{v_i+v_j}{2}=\frac{a_i+a_j}{2}\oplus\frac{b_i+b_j}{2}$.
			\item Then, return the 4 triangles $\{v_1, v_{12}, v_{13}\}$, $\{v_2,v_{12},v_{23}\}$, $\{v_3,v_{13},v_{23}\}$ and $\{v_{12},v_{23},v_{13}\}$. 		
		\end{enumerate}
		Notice that this process still produces a triangulation where all the vertices are rational points of $\|\Z_m*C_2\|$. Moreover, it clearly respects the $\Z_m$-action and when restricted to $\|C_m\|$, it simply halves each segment. 
		\item We then follow steps 4-9 of the \textit{general approach}. In our examples for $m=2, \dots, 6$, we took values of $k=3$ or $4$. 
	\end{enumerate}
	
	Since $\|Z_m*C_m\|$ can be seen as $m$ independent disks identified by their boundaries, we can plot the vertices of $T$ into $m$ unit disks with centers at the points $(R\cos(2\pi k/m), R\sin(2\pi k/m))$ for some large radius $R$. Then, we color a fine mesh of points in the disks with the color of their closest vertex to produce Figure~\ref{fig:upper_bounds}. 
	
	\subsubsection*{3D Case for $\bm{\Z_3}$}
	
	So far, for the 3-dimensional case we have only been able to compute $\cov{\Z_3}{3}$. For this we use the $E_2\Z_3$ space and the $\Z_2$-cover described before. Just as before, we can think of labeling all its rational points by elements of the module $\Q[\Z_m]$. We then apply the \textit{general approach} to find a $\Z_3$-cover of $E_3\Z_3=\Z_3*E_2\Z_3=(\Z_3^{*2})*C_3$, by choosing an appropriate subdivision procedure. For this, we subdivide each tetrahedron into 4 \textit{medial} tetrahedra and one central octahedron, that we further subdivide into 8 tetrahedra. That is, we proceed as follows:
	\begin{enumerate}
		\item Take a tetrahedron $\{v_1, v_2, v_3, v_4\}$. 
		\item For each pair of indices $1\leq i<j\leq 4$, define the point $v_{ij}=\frac{v_i+v_j}{2}$.
		
		Define the point $w=\frac{v_1+v_2+v_3+v_4}{4}$.  
		\item Then, return the 12 tetrahedra:
		\begin{align*}
			\{v_1,v_{12},v_{13},v_{14}\}, & 
			\{v_{2}, v_{12}, v_{23}, v_{24}\},  \{v_{3},v_{13},v_{23},v_{34}\}, \{v_4,v_{14},v_{24},v_{34}\} \\
			\{v_{12},v_{13},v_{14},w\}, & 
			\{v_{12},v_{23},v_{24},w\},
			\{v_{13},v_{23},v_{34},w\},
			\{v_{14},v_{24},v_{34},w\},\\
			\{v_{12},v_{14},v_{24},w\},&
			\{v_{12},v_{13},v_{23},w\},
			\{v_{13},v_{14},v_{34},w\},
			\{v_{23},v_{24},v_{34},w\}
		\end{align*}
	\end{enumerate}
	Notice that this procedure restricted to $E_2\Z_3$ is just the medial triangle subdivision used in the 2-dimensional case. Moreover, $\Delta K$ is still a $\Z_3$-simplicial complex. Thus we can finish this case by following steps 4-9 of the \textit{general approach}. In our computations, we use $k=3$, which produces a graph with 9,129 vertices, however, solving the related integer programming problem in \cite{gurobi} is almost immediate. 
	

	\bibliographystyle{halpha-abbrv}
	{\small
	\bibliography{References}}
\end{document}